\documentclass[11pt,a4paper]{amsart}
\makeatletter
\usepackage[english]{babel}
\usepackage[utf8]{inputenc}
\usepackage[square,sort,comma,numbers]{natbib}
\usepackage{graphicx}
\usepackage{amsfonts}
\usepackage{amsthm}
\usepackage{sansmath}
\usepackage{hyperref}
\usepackage[textsize=scriptsize]{todonotes}
\usepackage{fancyhdr}
\usepackage{amsmath,amssymb,euscript}
\usepackage{xparse}
\usepackage{tikz-cd}
\usepackage{verbatim}
\usepackage{mathtools}
\usepackage[new]{old-arrows}
\setcitestyle{square}

\numberwithin{equation}{section}

\theoremstyle{plain}
	\newtheorem{theo}[equation]{Theorem}
	\newtheorem{prop}[equation]{Proposition}
	\newtheorem{lemm}[equation]{Lemma}

	\newtheorem{coro}[equation]{Corollary}

	\newtheorem{lem/defn}[equation]{Lemma/Definition}

\theoremstyle{definition}
	\newtheorem{defi}[equation]{Definition}
        \newtheorem{term}[equation]{Terminology}
        \newtheorem{rema}[equation]{Remark}
	\newtheorem*{ackn}{Acknowledgements}
    	
    	\newtheorem*{notationandconvention}{Notation and Conventions}

	\newtheorem{ques}[equation]{Question}
        \newtheorem{conj}[equation]{Conjecture}
	\newtheorem{cons}[equation]{Construction}
        \newtheorem{observation}[equation]{Observation}
	
	\newtheorem{nota}[equation]{Notation}

\theoremstyle{remark}

\def\nc{\newcommand}
\def\on{\operatorname}
\def\co{\colon\thinspace}

\usepackage{todonotes}

\newcommand*{\rom}[1]{\expandafter\@slowromancap\romannumeral #1@}
\newcommand{\mmod}{/\!\!/}
\newcommand{\sslash}{\mmod}
\nc{\edit}[1]{\marginpar{\footnotesize{#1}}}
\newcommand{\lv}{\lvert}
\newcommand{\rv}{\rvert}

\nc{\Z}{\mathbb{Z}}
\nc{\z}{\mathbb{Z}}
\nc{\zp}{\mathbb{Z}_p}
\nc{\zpl}{\mathbb{Z}_{(p)}}
\nc{\PP}{\mathbb{P}}
\nc{\R}{\mathbb{R}}
\nc{\ot}{\otimes}
\nc{\I}{\mathbb{I}}
\nc{\f}{\mathbb{F}}
\nc{\fp}{\mathbb{F}_p}
\nc{\hz}{H\mathbb{Z}}
\nc{\hzp}{H\mathbb{Z}_p}
\nc{\hzpl}{H\mathbb{Z}_{(p)}}
\nc{\fq}{\mathbb{F}_q}
\nc{\fpn}{\mathbb{F}_{p^n}}
\nc{\CP}{\mathbb{CP}}

\nc{\wfq}{W(\mathbb{F}_q)}
\nc{\wfpn}{W(\mathbb{F}_{p^n})}

\nc{\pmot}{\frac{p-1}{2}}

\nc{\sph}{\mathbb{S}}

\nc{\sphtriv}{\mathbb{S}^{\on{triv}}}
\nc{\sphp}{\mathbb{S}_{p}}
\nc{\sphpl}{\mathbb{S}_{(p)}}

\nc{\sphwq}{\mathbb{S}_{W(\mathbb{F}_q)}}
\nc{\sphwpn}{\mathbb{S}_{W(\mathbb{F}_{p^n})}}
\nc{\sphwk}{\mathbb{S}_{W(k)}}
\nc{\sphxk}{\sph[x_k]}
\nc{\sphxm}{\sph[x_{m}]}
\nc{\sphxmk}{\sph[x_{mk}]}
\nc{\sphsk}{\sph[\sigma_k]}
\nc{\sphsmk}{\sph[\sigma_{mk}]}
\nc{\sphsell}{\sph[\sigma_\ell]}
\nc{\sphpsk}{\sph_{(p)}[\sigma_k]}
\nc{\sphpzk}{\sph_{(p)}[z_k]}
\nc{\sphpzpmone}{\sph_{(p)}[z_{2(p-1)}]}
\nc{\sphpztwo}{\sph_{(p)}[z_2]}
\nc{\sphpsmk}{\sph_{(p)}[\sigma_{mk}]}
\nc{\sphpsell}{\sph_{(p)}[\sigma_\ell]}

\nc{\sphzmk}{\sph[z_{mk}]}
\nc{\sphzk}{\sph[z_k]}
\nc{\sphzpmone}{\sph[z_{2(p-1)}]}
\nc{\sphztwo}{\sph[z_2]}
\nc{\musgmpn}{MU[\sigma_{2p^n-2}]}
\nc{\muallspi}{MU[\sigma_{2p^i-2} \mid i< n]}
\nc{\muallspinpo}{MU[\sigma_{2p^i-2} \mid i< n+1]}
\nc{\wdgmuallspi}{\wedge_{MU[\sigma_{2p^i-2} \mid i< n]}}
\nc{\thhmuallspi}{\ensuremath{\textup{THH}^{MU[\sigma_{2p^i-2} \mid i<n]}}}
\nc{\musi}{MU[\sigma_i]}
\nc{\musk}{MU[\sigma_k]}
\nc{\muspi}{MU[\sigma_{2p^i-2}]}

\nc{\sgmk}{\sigma_k}
\nc{\sgmmk}{\sigma_{mk}}
\nc{\sgml}{\sigma_{\ell}}
\nc{\sgmpi}{\sigma_{2p^i-2}}

\nc{\xrtmx}{X(\sqrt[m]{x})}
\nc{\artma}{A(\sqrt[m]{a})}

\nc{\hzplsk}{H\mathbb{Z}_{(p)}[\sigma_{k}]}

\nc{\hzplsmk}{H\mathbb{Z}_{(p)}[\sigma_{mk}]}

\nc{\wdgsphsmk}{\wdg_{\sph[\sigma_{mk}]}}
\nc{\wdgsphsk}{\wdg_{\sph[\sigma_{k}]}}
\nc{\wdgsphpsmk}{\wdg_{\sph_{(p)}[\sigma_{mk}]}}
\nc{\wdgsphpsk}{\wdg_{\sph_{(p)}[\sigma_{k}]}}

\nc{\bpn}{BP \langle n \rangle}
\nc{\hfp}{H\mathbb{F}_p}
\nc{\T}{\mathbb{T}}
\nc{\vo}{V(1)}
\nc{\vos}{V(1)_*}
\nc{\ttw}{T(2)}
\nc{\pis}{\pi_*}
\nc{\ttws}{T(2)_*}
\nc{\gdna}{\gdn(A\hmod)}
\nc{\xtn}{x_{2n}}
\nc{\gdnr}{\gdn(R\hmod)}

\nc{\wdg}{\wedge}

\nc{\wdgp}{\wedge_{\mathbb{S}_p}}
\nc{\wdgfp}{\wedge_{H\mathbb{F}_p}}
\nc{\wdgmu}{\wedge_{MU}}

\nc{\AAA}{\mathbb{A}}
\nc{\LL}{\mathbb{L}}
\nc{\OO}{\mathcal{O}}

\nc{\X}{\EuScript{X}}
\nc{\sZ}{\EuScript{Z}}

\nc{\id}{{\on{id}}}
\nc\cone{{\on{cone}}}
\nc{\Rep}{{\on{Rep}}}
\nc\Ob{{\on{Ob}}}
\nc\Spec{{\on{Spec}}}
\newcommand{\Mod}{\mathrm{Mod}}
\mathchardef\mhyphen="2D
\newcommand{\hmod}{\mhyphen\mathsf{Mod}}
\nc\coMod{{\on{coMod}}}
\nc\Perf{{\on{Perf}}}
\nc\End{{\on{End}}}
\nc{\into}{\hookrightarrow}
\nc{\tr}{\on{tr}}
\nc{\ev}{\on{ev}}
\nc{\im}{\on{im}}
\nc{\hfps}{H{\mathbb{F}_p}_*}
\nc{\Mot}{\on{Mot}}
\nc{\pt}{\on{pt}}
\nc{\coker}{\on{coker}}
\nc{\rk}{\on{rank}}
\nc{\TOP}{\on{Top}_{\mathbb{C}}^{s}}
\nc{\gr}{\on{Gr}}
\nc{\Catperf}{\text{Cat}^{\text{perf}}}
\nc{\Sym}{\on{Sym}}
\nc{\xra}{\xrightarrow}
\nc{\lra}{\xleftarrow}
\nc{\Bet}{\mathbf{Betti}_{X}}
\nc{\codim}{\on{codim}}
\nc{\Fred}{\on{Fred}}
\nc{\colim}{\on{colim}}
\nc{\KK}{{\bf K}}

\nc{\onto}{\twoheadrightarrow}
\nc{\A}{\mathbb{A}}
\nc{\Aff}{\on{Aff}}
\nc{\SH}{\on{SH}}
\nc{\QCoh}{\on{QCoh}}
\nc{\Alg}{\on{Alg}}
\nc{\alg}{\on{Alg}}
\nc{\lmod}{\on{LMod}}
\nc{\snm}{S_{2n}^m}
\nc{\sn}{S_{2n}}
\nc{\tnm}{T_{2n}^m}
\nc{\tn}{T_{2n}}
\nc{\rmod}{\on{RMod}}
\nc{\Br}{\on{Br}}
\nc{\ta}{\widetilde{\a}}
\nc{\Shv}{\on{Shv}}
\nc{\GG}{\mathbb{G}}
\nc{\red}{\color{red}}
\nc{\green}{\color{orange}}
\nc{\blue}{\color{blue}}
\nc{\an}{\on{an}}
\nc{\Pre}{\on{Pre}}
\nc{\assact}{\on{Ass}_{\on{act}}^{\otimes }}
\nc{\spact}{\on{Sp}_{\on{act}}^{\otimes }}
\nc{\spactprod}{{\on{Sp}^{\Z}_{\on{act}}}^{\otimes }}
\nc{\qc}{\on{qc}}
\nc{\op}{\on{op}}
\nc{\shEnd}{{\mathcal End}}
\nc{\Sph}{\mathbb{S}}
\nc{\Top}{\on{Top}}
\nc{\Map}{\on{Map}}
\nc{\Vect}{\on{Vect}}
\nc{\holim}{\on{holim}}
\nc{\fun}{\on{Fun}}

\nc{\leqgrsp}{\ensuremath{\textup{Leq}}\big(\grm(\sp^{BS^1}),(-)^{tC_p}\big)}
\newcommand{\cat}[1]{\ensuremath{\EuScript #1}}

\DeclareMathOperator{\map}{\ensuremath{\textup{Map}}}

\DeclareMathOperator{\Mor}{\ensuremath{\textup{map}}}

\DeclareMathOperator{\thh}{\ensuremath{\textup{THH}}}

\DeclareMathOperator{\ho}{\ensuremath{\textup{Ho}}}

\DeclareMathOperator{\grm}{\ensuremath{\textup{Gr}}_m}

\DeclareMathOperator{\hh}{\ensuremath{\textup{HH}}}

\nc{\C}{\cat C}
\nc{\D}{\cat D}
\nc{\V}{\cat V}

\def\A{\mathcal{A}}

\def\a{\alpha}

\def\Perf{\on{Perf}}

\def\sp{\on{Sp}}

\def\sp{\on{Sp}}
\def\fib{\on{fib}}
\def\V{\EuScript{V}}

\nc{\W}{\mathbb{W}}

\def\QCoh{\on{QCoh}}

\nc{\bE}{\mathbb{E}}
\nc{\An}{\mathrm{An}}
\nc{\Grp}{\mathrm{Grp}}
\nc{\Mon}{\mathrm{Mon}}
\nc{\Sp}{\mathrm{Sp}}
\nc{\Ztr}{\mathfrak{Z}}
\nc{\Fun}{\mathrm{Fun}}
\nc{\Cobar}{\mathrm{Cobar}}
\nc{\Barc}{\mathrm{Bar}}
\nc{\Ab}{\mathrm{Ab}}
\nc{\Q}{\mathbb{Q}}
\nc{\MU}{\mathrm{MU}}

\theoremstyle{plain}
\newcounter{zaehler}

\newtheorem{introthm}[zaehler]{Theorem}
\newtheorem{introcor}[zaehler]{Corollary}

\title{Towards the classification of DGAs with polynomial homology} 
\author[H.~\"O.~Bay\i nd\i r]{Haldun \"Ozg\"ur Bay{\i}nd{\i}r} 
\address{Departament de Matemàtiques i Informàtica, Universitat de Barcelona, Gran Via 585, 08007 Barcelona, Spain}
\email{ozgurbayindir@gmail.com}

\author[M.~Land]{Markus Land}
\address{Mathematisches Institut, Ludwig-Maximilians-Universit\"at M\"unchen, Theresienstra\ss e 39, 80333 M\"unchen, Germany}
\email{markus.land@math.lmu.de}

\usepackage{natbib}
\usepackage{graphicx}

\begin{document}

\begin{abstract}
We study the classification of $\mathbb{Z}$-DGAs   with polynomial homology $\mathbb{F}_p[x]$  with $\lvert x \rvert >0$, motivated by computations in algebraic $K$-theory. This classification problem was left open in work of  Dwyer, Greenlees, and Iyengar.
We prove that there are infinitely many such DGAs  for even $\lvert x \rvert$ and that for $\lvert x \rvert \geq 2p-2$ any such DGA is formal as a ring spectrum. Through this, we obtain examples of triangulated categories with infinitely many DG-enhancements and a classification of prime DG-division rings.

Combining our results with earlier work of the second author and Tamme, we obtain new (relative) algebraic $K$-theory computations for rings such as the mixed characteristic coordinate axes $\mathbb{Z}[x]/px$ and the group ring $\mathbb{Z}[C_{p^n}]$.
\end{abstract}

\maketitle

\tableofcontents

\section{Introduction}

In this paper, we study the classification of differential graded algebras, DGAs for short, whose homology is a polynomial algebra over $\fp$ (or more generally over $\Z/m$) on a single generator in a positive degree. Our motivation originally stems from the goal to perform explicit computations in algebraic $K$-theory, where work of the second author and Tamme \cite{land2023kthrypushouts} gives a number of examples of such DGAs whose algebraic $K$-theory is closely related to the algebraic $K$-theory of certain ordinary rings. We obtain such computations at the end of this paper.
However, the core of this paper consists of general results about DGAs with polynomial homology.

Let us first discuss what is known about the classification of DGAs with polynomial homology $\fp[x_k]$ with $\lvert x_k \rvert= k$.\footnote{The subscripts of the generators in graded rings will always denote the homological degree.} This is a natural question in homological algebra and studied in work of Dwyer, Greenlees, and Iyengar \cite{DGI} as we explain in more detail below. 

To set the stage, let us briefly lay out the setup which we will work in, namely in
the $\infty$-category of DGAs, which is obtained from the 1-category of DGAs by formally inverting quasi-isomorphisms of DGAs, i.e.\ maps that induce isomorphism in homology. As we work in this $\infty$-categorical setting, when we mention equivalences or uniqueness of DGAs or when we say there are infinitely many DGAs etc.\ we mean up to quasi-isomorphism unless stated otherwise; similarly, maps, tensor products and (co)limits are always understood in the derived sense.  

By \cite{shipley2007hzalgaredgas, lurie2016higher}, the $\infty$-category of DGAs identifies with $\Alg_\Z(\Sp)$, the $\infty$-category of $\bE_1$-$\Z$-algebras in the $\infty$-category $\Sp$ of spectra.\footnote{We suppress notation for the fully faithful, lax symmetric monoidal functor $\Ab \to \Sp$, often referred to as the Eilenberg--Mac Lane functor.}  Via the forgetful functor $\Alg_\Z(\Sp) \to \Alg(\Sp)$, which takes a DGA to its underlying ring spectrum, there is also the notion of topological equivalences between DGAs: two DGAs are said to be \emph{topologically equivalent} if their underlying ring spectra are equivalent. It follows that quasi-isomorphic DGAs are topologically equivalent, however conversely, there are examples of DGAs that are topologically equivalent but not quasi-isomorphic \cite{dugger2007topological}.

Importantly for us, a graded ring can and will be viewed as a DGA by equipping it with trivial differentials.  A DGA is called \emph{formal} if it is quasi-isomorphic to its homology considered as a graded ring (and hence as a DGA as just described) and \emph{topologically formal} if its underlying ring spectrum is equivalent to its homology, again viewed as a graded ring.

Since a number of invariants of DGAs, including algebraic $K$-theory, are invariants of the underlying ring spectrum, we follow the philosophy of Dugger and Shipley \cite{dugger2007topological} and  study the classification of DGAs up to quasi-isomorphism as well as up to topological equivalence. 

\subsection{Previous results}
Using a Koszul duality argument, Dwyer, Greenlees, and Iyengar \cite{DGI} prove the following:

\begin{prop}\label{prop intro Koszul duality poly to exterior}
    Let  $k\neq 0,1$. There is a canonical bijection between 
\begin{enumerate}
    \item[-] quasi-isomorphism classes of DGAs with homology $\Lambda_{\fp}[x_{k}]$ and
    \item[-] quasi-isomorphism classes of DGAs with homology $\fp[x_{-k-1}]$
\end{enumerate}
The same also holds for topological equivalence classes  in place of quasi-isomorphism classes.
\end{prop}

 Earlier, in \cite[Example 3.15]{dugger2007topological}, Dugger--Shipley classified DGAs with homology $\Lambda_{\fp}[x_k]$ for  $k>0$. By viewing them as square-zero extensions of $\fp$ by $\fp$, they deduce that there is a unique such DGA if $k>0$ is odd and that there are two quasi-isomorphism classes if $k>0$ is even and these two are topologically equivalent if $k\geq 2p-2$. The case $k<0$ is more complicated as such DGAs may not be given by square-zero extensions. Nevertheless, the case $k=-1$ is the main result of the work of Dwyer, Greenlees, and Iyengar \cite{DGI}. Using again a Koszul duality argument they prove:
\begin{theo}
    There is canonical bijection between
\begin{enumerate}
    \item[-] equivalence classes of DGAs with homology $\Lambda_{\fp}[x_{-1}]$ and
    \item[-] isomorphism classes of complete discrete valuation rings with residue field $\fp$.
\end{enumerate}
Here, equivalence refers to quasi-isomorphism or topological equivalence. 
\end{theo}

It follows that there are countably infinitely many DGAs with homology $\Lambda_{\fp}[x_{-1}]$. For the classification of DGAs with homology $\Lambda_{\fp}[x_{k}]$, the remaining case is therefore the case $k<-1$. Equivalently, what remains is the classification of DGAs with polynomial homology $\fp[x_{k}]$ for $k>0$ and the authors of \cite{DGI} leave this problem open. 

This classification question was the subject of earlier work of the first author \cite{bayindir2019dgaswithpolynomial}, in which the main result states that there is a unique non-formal DGA with homology $\fp[x_{2p-2}]$ and a non-formal $(2p-2)$-Postnikov section, providing the first example of a non-formal DGA with polynomial homology $\fp[x_k]$ with $k>0$ in the literature. Around the same time, Irakli Patchkoria also constructed a non-formal DGA given by the DGA quotient $\z \sslash p$, whose homology is  $\fp[x_2]$. Incidentally, $\Z\sslash p$ also appears in \cite[Ex.\ 4.31]{land2023kthrypushouts} as the $\odot$-ring (first introduced in \cite{LT}) associated to the Milnor square describing ${\Z \times_{\fp} \Z}$. In \cite{davis2023cyclic} the authors compute the negative cyclic homology of $\z \sslash p$ and in his dissertation, Julius Frank also studied the classification of DGAs with polynomial homology and proved that $\z \sslash p$ is not even topologically formal for $p>2$, whereas the second author and Tamme \cite[Remark 4.33]{land2023kthrypushouts} proved that $\z \sslash 2$ is in fact topologically formal.

\subsection{Classification results} 
Let us now summarise our main results in regards to the classification problem alluded to above. First, we discuss under what circumstances we can show that a DGA is formal.  We stress again that maps of DGAs always refers to derived maps, that is, maps in the $\infty$-category of DGAs as described above and that the term quasi-isomorphism is used for an equivalence in the $\infty$-category of DGAs. In what follows, let $m>1$ be an integer and $p$ be a prime.

\begin{introthm}[Formality]\label{ThmA}
      Let $n>0$ and $A$ be a DGA.
    \begin{enumerate}
        \item Assume that the homology of $A$ is $\Z/m[x_{2n}]$. If there is a map  $\Z/m \to A$ of DGAs, then $A$ is formal.
        \item Assume that the homology of $A$ is $\fp[x_{2n}]$. If there is a map $\fp \to A$ of ring spectra, then $A$ is topologically formal. 
        
        Moreover, if $\tau_{\leq 2p-4}A$ is topologically formal,  then there exists a map $\fp \to A$ of  ring spectra. In particular, $A$ is topologically formal if and only if $\tau_{\leq 2p-4} A$ is. As a result, $A$ is topologically formal if  $n \geq p-1$.
    \end{enumerate}
   
\end{introthm}

The final statement of Theorem~\ref{ThmA} in fact generalizes to the odd degree generator case:
\begin{introthm}[Topological formality]\label{ThmB}
        Let $n\geq 2p-2$. Every DGA with homology $\fp[x_{n}]$ is topologically formal.
\end{introthm}

\begin{rema}
    Both Theorems \ref{ThmA} and \ref{ThmB} in fact hold true more generally in case the homology of $A$ is a truncated polynomial algebra $\Z/m[x_{2n}]/x_{2n}^k$ for any $k>0$. 
\end{rema}

This fully resolves the topological classification of DGAs with (truncated) polynomial homology over $\fp$ in a sufficiently large degree generator, and in particular with exterior homology over $\fp$ in a sufficiently small degree generator. More precisely, equivalent to Theorem~\ref{ThmB} is the statement that every DGA with homology $\Lambda_{\fp}[x_n]$  is topologically formal whenever $n< -(2p-2)$ (Corollary \ref{coro later exterior dgas are topologically formal}).
\newline

Theorems~\ref{ThmA} and \ref{ThmB} above say nothing about the existence (and uniqueness) of non-formal DGAs with polynomial homology. Our next result remedies this. To state it, we need to digress briefly: For $m>1$ and $n>0$, in the body of the text we construct canonical DGAs $S_{2n}^m$ in an inductive manner (over $n$) whose homology is $\Z/m[x_{2n}]$. These DGAs are in fact also essential in the proof of Theorems \ref{ThmA} and \ref{ThmB}. Moreover, for these DGAs, we show that it is possible to adjoin suitable roots of the polynomial generator; we explain this in some more detail in Section~\ref{sec:proof-ingredients} below. In particular, for each $l\geq 1$, we construct DGAs $S_{2nl}^m\big[\sqrt[l]{x_{2nl}}\big]$ whose homology is isomorphic to $\Z/m[x_{2n}]$, see Construction~\ref{cons adj root to snm} for details.

\begin{introthm}[Existence]\label{ThmC}
    Let $n>0$ and $m>1$ and $p$ be a prime. 
    
    \begin{enumerate}
        \item  The collection $\{S_{2nl}^m \big[\sqrt[l]{x_{2nl}}\big]\}_{l\geq 1}$ consists of pairwise non-quasi-isomorphic DGAs.
        In particular, up to quasi-isomorphism, there are infinitely many pairwise distinct DGAs with homology $\Z/m[x_{2n}]$.
        \item For $l \geq \frac{p-1}{n}$, the DGA $S_{2nl}^p \big[\sqrt[l]{x_{2nl}}\big]$ is topologically equivalent to $\fp[x_{2n}]$, i.e.\ is topologically formal. 
    \end{enumerate}
\end{introthm}

Consequently, for $n>0$ we also obtain infinitely many pairwise distinct DGAs with homology $\Lambda_{\fp}[x_{-2n-1}]$, where all but finitely many are topologically equivalent to $\Lambda_{\fp}[x_{-2n-1}]$. 
To the best of our knowledge, these are the first examples of infinitely many pairwise non-quasi-isomorphic DGAs that are all topologically equivalent, i.e.\ infinitely many pairwise distinct $\Z$-algebra structures on a single ring spectrum. We will later also leverage this result to construct exotic dg enhancements of certain triangulated categories, see Section~\ref{sec:intro:dg-enhancements} in this introduction.

\begin{rema}
Let us remark on the case where we are given a DGA $A$ with homology $\z/p^s[x_{2n}]$ for $n>0$ and $s\geq 3$ or $s\geq 2 $ for $p$ odd. By the (surprising) recent results of Burklund \cite{burklund2022multiplicative}, $\sph/p^s$ is a ring spectrum (where $\sph$ denotes the sphere spectrum). If there is a map of ring spectra $\z/p^s \to A$, then the adjoint of the canonical composite
\[\sph/p^s \to \z/p^s \to A\]
is a DGA map $\z/p^s\to A$. From Theorem \ref{ThmA}, we deduce that $A$ is topologically formal if and only if it is formal. More generally, we prove (Proposition~\ref{prop infty many dgas up to topological equiv}):
\begin{enumerate}
   \item[(1)] The collection $\{S_{2nl}^{p^s} \big[\sqrt[l]{x_{2nl}}\big]\}_{l\geq 1}$ consists of pairwise non-topologically-equivalent DGAs with homology $\Z/p^s[x_{2n}]$.
\end{enumerate}

In particular, there are infinitely many topological equivalence classes of DGAs with homology $\Z/p^s[x_{2n}]$. Since $\MU/m$ is an $\MU$-algebra \cite{angeltveit2008thhofainftyringspectra} for each $m>1$, the same arguments prove the following (Proposition~\ref{prop infty many dgas up to topological equiv}).
\begin{enumerate}
    \item[(2)] The collection $\{S_{2nl}^m \big[\sqrt[l]{x_{2nl}}\big]\}_{l\geq 1}$ consists of DGAs that are  pairwise non-$\MU$-algebra-equivalent.
\end{enumerate}

In addition, we show that for $l,l'< \frac{p-1}{n}$, the DGAs $S_{2nl}^p\big[\sqrt[l]{x_{2nl}}\big]$ and $S_{2nl'}^p\big[\sqrt[l']{x_{2nl'}}\big]$ are topologically equivalent if and only if $l=l'$ and neither are topologically formal (Proposition \ref{prop root adj guys that are topologically distinct for all m}), showing that part (2) of Theorem~\ref{ThmC} above is sharp.
\end{rema}

The following is still open:

\begin{conj}
    There exist DGAs that are not quasi-isomorphic but equivalent as $\MU$-algebras.
\end{conj}

Also, we do not know whether every DGA with homology $\Z/p^s[x_{2k}]$ is equivalent to one of the form $S_{2nl}^{p^s}\big[\sqrt[l]{x_{2nl}}\big]$. In particular, Theorem~\ref{ThmC} is not a complete classification.

\begin{ques}
    Is every DGA with homology $\z/p^s[x_{2n}]$ ($n>0$) quasi-isomorphic to  a DGA of the form $S_{2nl}^{p^s}[\sqrt[l]{x_{2nl}}]$ for some $l$?
\end{ques} 

However, there is one situation in which we can say something in this direction:
\begin{introthm}[Uniqueness]\label{ThmD}
    Let $n>0$ and $p$ be a prime. The DGA $S_{2n}^p$ is the unique DGA with homology $\fp[x_{2n}]$ and non-formal $2n$-Postnikov section.
\end{introthm}

This result generalizes the previously mentioned earlier work of the first author \cite{bayindir2019dgaswithpolynomial} which treated the case $n=p-1$. 

\subsection{Proof Ingredients}\label{sec:proof-ingredients}
We now spell out the basic constructions and ideas that go into the proofs of the above results. After this, we end the introduction with a number of applications.

As indicated earlier, our results build on the construction of certain DGAs $S_{2n}^m$ whose homology is $\Z/m[x_{2n}]$. We briefly explain this construction.
For the rest of this paper, we denote the homology of a DGA $A$ by $\pis(A)$, as it is also the homotopy of the underlying ring spectrum.
\begin{nota}
Let $A$ be a DGA and $x\in \pi_k(A)$. We define $A\sslash x$ as the  pushout $\Z \amalg_{\Z[X_k]} A$ in the $\infty$-category of DGAs\footnote{I.e.\ in more classical terminology a homotopy pushout of DGAs.} 
where $\Z[X_k]$ is the free DGA on a generator of degree $k$ which is sent to $0$ in $\Z$ and to $x$ in $A$.
\end{nota}
Then we set $S_2^m := \Z\sslash m$ which has homology $\Z/m[x_2]$ as indicated above and inductively define
 \[S_{2n}^{m} := S_{2n-2}^m\sslash x_{2n-2}.\]
 Part of this definition is, of course, to show that indeed $\pis(S_{2n}^m) =\Z/m[x_{2n}]$. We then obtain a sequence of DGAs
 \[ S_2^m \to S_4^m \to \dots \to S_{2n}^m \to \dots \]
 whose colimit is $\Z/m$.
By construction, we therefore obtain an odd cell decomposition of each $S_{2n}^m$ and consequently also of $\colim_n S_{2n}^m \simeq \z/m$ (the reader may want to contrast this with the Hopkins--Mahowald theorem stating that $\fp$ is obtained from $\sph$ by attaching a single $\bE_2$-cell in dimension $1$). This gives obstructions to the existence of maps $S_{2n}^m \to A$ for another DGA $A$. If $\pis(A)$ is concentrated in even degrees, this also implies that if there is a map $S_{2n}^m\to A$ (or $\z/m\to A$) then it is unique up to homotopy.  We then show that if $A$ has homology $\Z/m[x_{2n}]/x_{2n}^k$, then there is a map $S_{2l}^m \to A$ carrying the generator $x_{2l}$ to a non-trivial element in $\pi_{2l}(A)$ if and only if $l$ is the smallest integer such that $\tau_{\leq 2l}(A)$ is not formal (Theorem \ref{theo maps from snk to dgas with polynomial homology}). Theorem~\ref{ThmD} is an immediate consequence of this result.

For the proof of Theorem \ref{ThmA} (1), given a DGA map $\z/m\to A$, we would like to extend it to an equivalence $\z/m\otimes_{\z}\z[x_{2n}]\xrightarrow{\simeq} A$. If these were classical associative rings, such an extension exists if there is a map $\z[x_{2n}]\to A$ whose image commutes with the image of $\z/m$ in $A$ (which is, of course, automatic). This idea also applies to DGAs (and ring spectra) through the theory of centralizers \`a la Lurie; in the case at hand the centralizers identify with (topological) Hochschild cohomology. We then compute enough about the centralizer of the given map $\z/m\to A$ to run the above argument and obtain Theorem \ref{ThmA} (1). 

Moreover, we compute enough about the Hochschild cohomology of $S_{2n}^m$ (i.e.\ the centralizer of the identity of $S_{2n}^m$) as to construct a $\z[X_{2n}]$-algebra structure on $S_{2n}^m$ where $X_{2n}$ acts via $x_{2n}$ and then define for $l\geq 1$
\[S_{2nl}^m[\sqrt[l]{x_{2nl}}] := S_{2nl}^m \otimes_{\z[X_{2nl}]} \z[X_{2n}]\]
following \cite{ausoni2022adjroot} where $\z[X_{2nl}] \to \z[X_{2n}]$ carries $X_{2nl}$ to $X_{2n}^l$; this is then a DGA with homology $\z/m[x_{2n}]$. By considering the maps they receive from $S_{2nl}^m$, we deduce $S_{2nl}^m[\sqrt[l]{x_{2nl}}] \not \simeq S_{2nl'}^m[\sqrt[l']{x_{2nl'}}]$ for $l \neq l'$ which gives  Theorem \ref{ThmC} (1).

We then show that $p=0$ in the Hochschild cohomology of $S_{2p-2}^p$ and deduce from the Hopkins-Mahowald theorem mentioned earlier that there is a map of ring spectra $\fp \to S_{2p-2}^p$. By our earlier results, this implies the topological formality of $S_{2p-2}^p$. For a DGA $A$ with homology $\fp[x_{2n}]$ ($n>0$) and formal $(2p-4)$-Postnikov truncation, the odd cell decomposition of $S_{2p-2}^p$ gives a map $S_{2p-2}^p\to A$ providing the desired map $\fp\to A$ of ring spectra for Theorem \ref{ThmA} (2). Theorem \ref{ThmB} (in the case of odd degree generators) is obtained similarly by  analyzing the centralizer of the map $S_{2p-2}^p \to A$ itself.
\newline

We finish this introduction with a number of applications of the aforementioned classification results. 

\subsection{Exotic DG-enhancements} \label{sec:intro:dg-enhancements}
Triangulated categories are ubiquitous in many areas such as representation theory, homotopy theory, and algebraic geometry. However, often it is advantaguous (or necessary) to enhance a triangulated category with a higher categorical structure; classically these arise in the form of DG-enhancements and after the work of Lurie in the form of stable $\infty$-categories. In this language, a DG-enhancement amounts to equipping a stable $\infty$-categorical lift with a $\z$-linear structure. It is then natural to ask whether a given triangulated category admits an enhancement, and if so, how many. For instance, in \cite{MSS, RvdB} triangulated categories without enhancements are constructed, and in \cite{schlichting2002kthryoftriangulated,dugger2009curious,rizzardo2019nonuniqueenhancements} examples of triangulated categories with non-unique DG-enhancements are constructed.

Here, we obtain (to our knowledge) the first example of a  triangulated category with infinitely many distinct DG-enhancements. 
Let  $F_{2n}^p(l)$ denote the localization $S_{2nl}^p[\sqrt[l]{x_{2nl}}][x_{2nl}^{-1}]$; a DGA with homology $\fp[x_{2n}^{\pm1}]$. As an application of Theorem~\ref{ThmC}, we obtain:
\begin{introcor}\label{CorE}
The collection 
\[\{\Mod(F_{2n}^p(l))\}_{l\geq \frac{p-1}{n}}\] 
consists of pairwise distinct $\Z$-linear structures on the stable $\infty$-category $\Mod(\fp[x_{2n}^{\pm1}])$. In particular, the triangulated category $\ho(\Mod(\fp[x_{2n}^{\pm1}]))$ admits infinitely many pairwise distinct DG-enhancements.
\end{introcor}

\subsection{Classification of prime DG-fields} The fields $\fp$ and $\Q$ are prime fields: Any map of fields with codomain $\fp$ or $\Q$ is an isomorphism. 
We reflect this idea at the level of DGAs: 
We say a DGA is a \emph{DG-division ring} (DGDR) if every non-zero homogeneous element in its homology is a unit (we warn, however, that our definitions are different than the recent ones given in \cite{zimmermann2024differential}). For instance, $F_{2n}^p := F_{2n}^p(1) = S_{2n}^p[x_{2n}^{-1}]$ is a DGDR. Furthermore, we say a DGDR $D$ is \emph{prime} if every map of DGDRs with codomain $D$ is an equivalence.

\begin{introcor}\label{CorF}
The collection of all prime DGDRs is given by 
    \[ \{ \Q,\fp, F_{2n}^p \;|\; p \  \textrm{prime and } n>0 \}.\]
 Every DGDR receives a map from at least one of the prime DGDRs. Furthermore, every DGDR with even homology receives a map from exactly one of the prime DGDRs and this map is unique up to homotopy.
\end{introcor}

\subsection{Applications to algebraic $K$-theory}
As a first application, we obtain a computation of the algebraic $K$-theory of the mixed characteristic coordinate axes $\Z[x]/px = \Z \times_{\fp} \fp[x]$:
\begin{introcor}\label{CorG}
    Let $p$ be a prime. Then we have
    \[ K_n(\Z[x]/px) = \begin{cases} K_n(\Z) \oplus \mathbb{W}_{\frac{n}{2}}(\fp) & \text{ if $n$ is even } \\ K_n(\Z) & \text{ if $n$ is odd} \end{cases}\]
    where $\mathbb{W}_r(\fp)$ denotes the ring of big Witt vectors of length $r$.
\end{introcor}

Indeed, a special case of \cite[Lemma 4.30]{land2023kthrypushouts} gives a pullback diagram
\[\begin{tikzcd}
    K(\z[x]/px) \ar[r] \ar[d] & K(\z) \ar[d] \\
    K(\fp[x]) \ar[r] & K(\odot)
    \end{tikzcd}\]
with $\pis(\odot) \cong \fp[t_2]$. By construction, there is a map $\fp \to \odot$ and by Theorem~\ref{ThmA}, we deduce that $\odot$ is the formal DGA $\fp[t_2]$. The $K$-theory of $\fp[t_2]$ is known due to \cite{BM} or independently due to \cite[Example 4.29]{land2023kthrypushouts} giving the above corollary.

As another application, we have the following. Let $l\geq 1$ and $f \in \Z[x]$ be a polynomial in $x^l$ with constant term $p$, e.g.\ $f=x^l-p$.
\begin{introcor}
    If $l\geq p-1$ in the situation above, there is a pullback diagram
    \[\begin{tikzcd}
        K(\Z[x]/xf) \ar[r] \ar[d] & K(\Z) \ar[d] \\
        K(\Z[x]/f) \ar[r] & K(\fp[t_2])
    \end{tikzcd}\]
\end{introcor}
We prove this by showing that the DGA $\odot$ obtained from \cite[Lemma 4.30]{land2023kthrypushouts} in the present situation is topologically formal, as an application of Theorem~\ref{ThmB} using a grading trick involving the assumption $l\geq p-1$. We note that, a priori, neither of the two maps $\Z[x]/f \to \odot \leftarrow \Z$ factors through the unit $\fp \to \fp[t_2]$.

Finally, we have:
\begin{introcor}\label{coro I}
    For each $n\geq 1$, there is a pullback square:
    \[\begin{tikzcd}
    K(\z[C_{p^n}]) \ar[r] \ar[d] & K(\z[\zeta_{p^n}]) \ar[d]\\
    K(\z[C_{p^{n-1}}]) \ar[r] & K(\fp[t_2]).
\end{tikzcd}\]
\end{introcor}
The case $n=1$ is \cite[Example 4.32]{land2023kthrypushouts} and independently due to Krause--Nikolaus - we do not reprove it here. Before applying $K$-theory, the square for $n=1$ maps to the square for general $n$; therefore, the $\odot$-ring for the case $n=1$ maps to the general one. In particular, for all $n$, we obtain a map $\fp \to \odot$ of ring spectra and deduce formality of $\odot$ by Theorem~\ref{ThmA} (2). 

\begin{rema}
    We remark that the above Corollaries about $K$-theory hold similarly for any localizing invariant $E$ of stable $\infty$-categories (most helpful if $E(\fp[t_2])$ has been computed).
\end{rema}

\begin{ackn}
    The authors would like to thank Joana Cirici and Oscar Randal-Williams  for useful discussions on this work.  The first author is  supported by the Beatriu de Pinós Programme (2023 BP 00043) of the Ministry of Research and Universities of the Government of Catalonia and both authors were supported by the Deutsche Forschungsgemeinschaft (DFG) grant ``Symmetrien in Topologie und Algebra" (527329998).
\end{ackn}

\begin{notationandconvention}

\begin{enumerate}
\item We work in the $\infty$-categorical setting, so all tensor products, maps, (co)limits, mapping spaces/spectra etc.\ are  derived. 
\item When we speak of DGAs, we shall always mean $\Z$-DGAs.
\item For a (commutative) ring $R$, there is a corresponding (commutative) ring spectrum characterized by the property that all of its homotopy groups are trivial except degree zero where it is given by $R$. This ring spectrum is often denoted by $HR$ but we drop $H$ from our notation and denote  $HR$ also by $R$. 
\item With this notation, for commutative $R$, the $\infty$-category $\Alg(R)$\footnote{The ordinary category of discrete $R$-algebras does not appear in this paper.} of $R$-algebra spectra is equivalent to the $\infty$-category of $R$-DGAs,  that is of differential graded $R$-algebras with quasi-isomorphisms formally inverted. We shall not distinguish between $R$-DGAs and $R$-algebra spectra for that reason. For $A$ an $R$-algebra spectrum, we also denote by $A$ its underlying ring spectrum. 
\item We denote the sphere spectrum by $\sph$, this is the monoidal unit of the $\infty$-category of spectra and hence, an $\sph$-algebra is the same thing as a ring spectrum.
\item For a DGA $A$, the homology ring of $A$ agrees with the homotopy ring $\pis A$ of the corresponding $\z$-algebra (or $\sph$-algebra). For this reason, we denote the homology of $A$ also by $\pis A$. 
\item We regard a graded ring as a DGA by equipping it with trivial differentials. DGAs which are equivalent to their homology (viewed as DGAs in the manner just mentioned) are called formal. 
\item For generators of homotopy rings (or homology rings), we will use subscripts to indicated the homotopical (or homological) degree.
\item The $k=\infty$ case of the truncated polynomial algebra $R[y]/y^{k}$ denotes the usual polynomial algebra $R[y]$.
\item Whenever we write $\z/m$, we assume $m>1$.

\end{enumerate}
\end{notationandconvention} 

\section{Quasi-isomorphism classes of DGAs with polynomial homology}\label{sect quasiiso classes of DGAs}

The first goal of this section is to construct the DGAs $S_{2n}^m$ mentioned in the introduction. To that end, we will first recall some basic results on the homology of $\Omega \CP^n$ which will be needed (\S\ref{sec:3.1}). Then we will construct the DGAs $S_{2n}^m$ and prove first basic properties about them (\S\ref{sec:3.2}). Upon proving our main formality/non-formality criteria for DGAs with polynomial homology (\S\ref{sec:3.3} and \S\ref{sect nonformality via maps from s guys}) and establishing the root adjunctions for $S_{2n}^m$, we prove Theorem \ref{ThmC} (1) in \S\ref{sec adj roots}.

\subsection{The homology of $\Omega\CP^k$}\label{sec:3.1}
To begin, let us recall the following well known results, starting with the fibre sequence
\[ S^1 \to S^{2n+1} \xrightarrow{p} \CP^n.\]
Note that it can be extended to the right once by the inclusion $\CP^n \to \CP^\infty$ and it can of course also be extended to the left to give a fibre sequence
\[ \Omega S^{2n+1} \xrightarrow{\Omega p} \Omega \CP^n \to S^1\]
in which the latter map induces an isomorphism on $\pi_1$ and identifies with the map $\Omega \CP^n \to \Omega \CP^\infty \simeq S^1$, showing that it is an $\bE_1$-map.

When $n=1$, we have $H_*(\Omega \CP^1;\Z) = H_*(\Omega S^2) \cong \Z[x_1]$ as $\Omega S^2$ is the free $\bE_1$-group on the pointed space $S^1$. For $n\geq 2$, since the above fibration sequence is on of $\bE_1$-spaces, the associated homological Serre spectral sequence is multiplicative. For degree reasons there are no differentials and no extension problems in this spectral sequence, and we obtain an isomorphism of rings
\[ H_*(\Omega \CP^n;\Z) \cong \Z[u_{2n}] \otimes_\Z \Lambda_\Z[e]\]
with $|e|=1$. By the same argument or by base-change, for any discrete commutative ring $R$ we thus obtain:

\begin{lemm}\label{lemma:pontryagin-ring-for-CPn}
For a discrete commutative ring $R$, there is an isomorphism of $R$-algebras 
\[ H_*(\Omega \CP^n;R) \cong \begin{cases} R[x_1] & \text{ for } n=1 \\ R[u_{2n}] \otimes_R \Lambda_R[e] & \text{ for } n \geq 2 .\end{cases}\]    
\end{lemm}

By construction, we also find the following.
\begin{lemm}\label{lemma:1-truncation-CP}
For all $n\geq 1$, the $\bE_1$-map $\Omega \CP^n \to S^1$ induces an $R$-algebra map
\[ R[\Omega \CP^n] \to R[S^1]\]
which exhibits the target as the Postnikov $1$-truncation of the source.    
\end{lemm}

Next, recall that for any $M$ in $\Mon(\An)$\footnote{We follow recent trends and denote the $\infty$-category of anima, spaces, $\infty$-groupoids by $\An$.} there is a functorial equivalence
\[ R \otimes_{R[M]} R \simeq R[BM] \]
since the functor $R[-] \colon \Mon(\An) \to \Alg(R)$ is monoidal and commutes with colimits; here $BM$ denotes the Bar construction of the monoid $M$. In particular, we obtain:
\begin{coro}\label{coro bar const on cpn}
The $\bE_1$-map $\Omega \CP^n \to \Omega \CP^\infty \simeq S^1$ induces a commutative diagram
\[\begin{tikzcd}
    R \otimes_{R[\Omega \CP^n]} R \ar[r] \ar[d,"\simeq"] & R \otimes_{R[S^1]} R \ar[d,"\simeq"]\\
    R[\CP^n] \ar[r] & R[\CP^\infty]
\end{tikzcd}\]
where the lower horizontal map is induced by the canonical inclusion $\CP^n \to \CP^\infty$.
\end{coro}

Finally, we will need the following. 
\begin{lemm}\label{lemma:pushout-loop-CPk}
    The diagram
    \[\begin{tikzcd}
        R[x_{2n}] \ar[r,"u_{2n}"] \ar[d] & R[\Omega \CP^n] \ar[d] \\
        R \ar[r] & R[\Omega \CP^{n+1}]
    \end{tikzcd}\]
    is a pushout in $\Alg_{\bE_1}(R)$. 
\end{lemm}
\begin{proof}
Note that $R[\Omega S^{2n+1}] \simeq R[x_{2n}]$ is the free $\bE_1$-$R$-algebra on a single generator $x_{2n}$ in degree $2n$ and that $H_{2n}(\Omega \CP^{n+1};R) = 0$. Under the equivalence $R[x_{2n}] \simeq R[\Omega S^{2n+1}]$, the upper horizontal arrow becomes the map induced by $\Omega p\colon \Omega S^{2n+1} \to \Omega \CP^n$ since its induced map on homology also hits $u_{2n} \in H_{2n}(\Omega \CP^n;R)$ by construction. The diagram under investigation is hence equivalent to the diagram obtained from the diagram 
\[ 
\begin{tikzcd}
    \Omega S^{2n+1} \ar[r,"\Omega p"] \ar[d] & \Omega \CP^n \ar[d,"\Omega i"] \\
    \ast \ar[r] & \Omega \CP^{n+1}
\end{tikzcd}
\]
upon applying the (left adjoint) functor $R[-]\colon \Mon(\An) \to \Alg(R)$. It hence suffices to prove that this diagram is a pushout in $\Mon(\An)$. Since it consists of group-like monoids, and the inclusion $\Grp(\An) \subseteq \Mon(\An)$ admits both adjoints and hence in particular preserves pushouts, it suffices to prove that it is a pushout in $\Grp(\An)$. Now, loop and Bar construction induce inverse equivalences $\Grp(\An) \simeq \An_*^{\geq1}$
between $\bE_1$-groups in anima and pointed connected anima, so the result follows from the well-known pushout:
\[\begin{tikzcd}
    S^{2n+1} \ar[r,"p"] \ar[d] & \CP^n \ar[d,"i"] \\
    \ast \ar[r] & \CP^{n+1}
\end{tikzcd}\qedhere\]
\end{proof}

\subsection{The DGAs $S_{2n}^m$}\label{sec:3.2}

Here, our goal is to construct the non-formal DGAs $S_{2n}^m$  with homology $\z/m[x_{2n}]$ that we mentioned in the introduction. We will do so inductively by forming appropriate pushouts of DGAs. We introduce the following notation:

\begin{nota}\label{cons z mod mod m}
Let $A$ be a DGA and $x\in \pi_k(A)$. We define a DGA $A\sslash x$ as the pushout of DGAs
\[\begin{tikzcd}
    \Z[X_{k}] \ar[r,"x"] \ar[d,"0"'] & A \ar[d] \\
    \Z \ar[r] & A\sslash x
\end{tikzcd}\]
where $\Z[X_k]$ is the free DGA on a generator of degree $k$, the top horizontal arrow classifies the element $x \in \pi_k(A)$  and the left vertical map classifies the $0$ element. 
\end{nota}

\begin{lemm}\label{lemm homology of z mod mod m}
There is 
\begin{enumerate}
    \item an isomorphism of graded rings $\pis (\z \sslash m) \cong \z/m[x_2]$, and
    \item a canonical equivalence of $\Z/m$-algebras $\Z/m \otimes_\Z (\Z\sslash m) \simeq \Z/m[\Omega \CP^1]$
\end{enumerate}

\end{lemm}
\begin{proof}
For the first claim, see \cite[Lemma 4.30]{land2023kthrypushouts} or \cite[Section 2]{davis2023cyclic} (the argument in loc.\ cit.\ applies verbatim to our case). For the second claim, note that the functor $\Z/m\otimes_\Z - \colon \Alg(\Z) \to \Alg(\Z/m)$ preserves colimits. Therefore, the induced diagram of $\Z/m$-DGAs
\[ \begin{tikzcd}
    \Z/m[X_0] \ar[r,"m"] \ar[d] & \Z/m \ar[d] \\
    \Z/m \ar[r] & \Z/m \otimes_\Z \Z\sslash m
\end{tikzcd}\]
is again a pushout. Since $m=0$ in $\Z/m$, this pushout is obtained by applying the free $\Z/m$-algebra functor to the diagram of $\Z/m$-modules $0 \leftarrow \Z/m \to 0$, showing that the above pushout is given by the free $\Z/m$-algebra $\Z/m[X_1] \simeq \Z/m[\Omega S^2] \simeq \Z/m[\Omega \CP^1]$ on the $\Z/m$-module $\Sigma \Z/m$.
\end{proof}

The following lemma will be the key input in our inductive definition of the DGAs $S_{2n}^m$ from the introduction.
\begin{lemm}\label{lemm pushout homology ring for s2k}
    Let $A$ be a DGA with homology $\Z/m[x_{2n}]$ equipped with an equivalence of $\Z/m$-algebras
    $\Z/m \otimes_{\Z}A \simeq \Z/m[\Omega \mathbb{CP}^n]$.
    Then there is 
    \begin{enumerate}
        \item a preferred equivalence of $\Z/m$-algebras $\Z/m\otimes_{\z}(A\sslash x_{2n})\simeq \Z/m[\Omega \mathbb{CP}^{n+1}]$
        \item and an isomorphism of graded rings $\pis\big(A\sslash x_{2n} \big) \cong \Z/m[x_{2(n+1)}]$.
    \end{enumerate}
\end{lemm}
\begin{proof}
Recall that $A\sslash x_{2n}$ is the pushout $\Z \amalg_{\Z[X_{2n}]} A$ and that the functor $\Z/m \otimes_\Z -\colon \Alg(\Z) \to \Alg({\Z/m})$ preserves colimits, in particular pushouts. Consequently, there is a preferred equivalence
\begin{align*}
    \Z/m \otimes_\Z (\z \amalg_{\Z[X_{2n}]} A) & \simeq  \Z/m\amalg_{\Z/m[X_{2n}]}(\Z/m \otimes_\Z A)  \\
    & \simeq \Z/m \amalg_{\Z/m[X_{2n}]} \Z/m[\Omega \CP^n]  \\
    & \simeq \Z/m[\Omega \CP^{n+1}]
\end{align*} 
where the final equivalence follows from
Lemma~\ref{lemma:pushout-loop-CPk} and the fact that the image of $x_{2n} $ under the map $\pi_{2n}(A) \to \pi_{2n}(\Z/m \otimes_\Z A)$ is a generator; this shows the first claim. 

Using this equivalence, we then consider the map of DGAs
\begin{equation}\label{comparison-map} A\sslash x_{2n} \to \Z/m \otimes_\Z (A\sslash x_{2n}) \simeq \Z/m[\Omega \CP^{n+1}]
\end{equation}
induced by the unit map $\Z \to \Z/m$. The ring map $A \to A\sslash x_{2n}$ shows that $m=0$ in $\pi_0(A\sslash x_{2n})$. In particular, the map \eqref{comparison-map} induces an injective ring homomorphism on graded homotopy groups and an equivalence of spectra $\Z/m[\Omega\CP^{n+1}] \simeq A\sslash x_{2n} \oplus \Sigma A\sslash x_{2n}$. Now recall from Lemma~\ref{lemma:pontryagin-ring-for-CPn} that $\pis(\Z/m[\Omega\CP^{n+1}]) = \Lambda_{\Z/m}[z_1] \otimes_{\Z/m} \Z/m[u_{2(n+1)}]$. It follows that the map \eqref{comparison-map} induces an isomorphism on $\pi_{2\ast}$, showing the second claim.
\end{proof}

Finally, we are ready to construct the DGAs $S_{2n}^m$.

\begin{cons}\label{cons of the dgas s2km}
We inductively define DGAs $S_{2n}^m$ with the following properties:
\begin{itemize}
    \item $\pis S_{2n}^m \cong\z/m[\xtn]$ as graded rings and 
    \item $\z/m \otimes_{\z} S_{2n}^m \simeq \z/m[\Omega \CP^n]$ as $\z/m$-algebras.
\end{itemize}
For the inductive start, we set 
\[S_{2}^m := \z\sslash m,\]
which satisfies the properties listed above by Lemma~\ref{lemm homology of z mod mod m}.
For the inductive step, we define  $S_{2n+2}^m$ to be $S_{2n}^m\sslash x_{2n}$ which satisfies the properties listed above by the induction hypothesis and an application of Lemma~\ref{lemm pushout homology ring for s2k}.
\end{cons}

By construction, we obtain a sequence of DGAs
\[\z \sslash m =S^m_2 \to S^{m}_4 \to \cdots \to S^m_{2n} \to S^m_{2n+2} \to \cdots \to \Z/m\]
whose colimit  over $n$ is $\Z/m$ as homotopy groups commute with filtered colimits and each map $\pis S_{2n}^m \to \pis S_{2n+2}^m$ is trivial on positive homotopy groups for degree reasons and an isomorphism on $\pi_0$.  

\begin{rema}\label{remark:extensions-to-S2n}
For a given DGA $B$ with $m=0$ in $\pi_0 B$, one obtains a map of DGAs $\z \sslash m \to B$ and such extensions are parametrized by $\pi_1(B)$. By the pushout description of $S_4^{m}$ above, this map extends to a DGA map $S_{4}^m\to B$ if and only if it carries $x_2 \in \pis(\z \sslash m)$ to zero  and again such extensions are parametrized by $\pi_3(B)$. This process continues inductively and provides lifts to $S^{m}_{2n}\to B$ whenever the previous generators are mapped to zero. If the homology of $B$ is concentrated in even degrees, all of these extensions are unique up to homotopy whenever they exist (see Corollary \ref{coro s2k maps are unique up to homotopy}) as we show next. 
\end{rema}

\begin{lemm}\label{lemm mapping spaces for pushout dgas}
    Let $A$ and $B$ be DGAs and let $x_{2n} \in \pi_{2n}(A)$ for some $n$. Assume that the homology of $B$ is concentrated in even degrees and that  
    $\textup{Map}_{\Alg(\Z)}(A,B)$
    has homotopy groups concentrated in odd degrees (in particular, it is connected). Then
    \[\textup{Map}_{\Alg(\Z)}(A\sslash x_{2n},B)\]
    is either empty or has homotopy groups concentrated in odd degrees (in particular, it is empty or connected). 
    Furthermore, the induced map 
    \[\pi_1 \Map_{\Alg(\Z)}(A\sslash x_{2n},B) \to \pi_1 \Map_{\Alg(\Z)}(A,B)\]
    is surjective.
\end{lemm}
\begin{proof}
Since $A\sslash x_{2n} = \Z \amalg_{\Z[X_{2n}]} A$ is a pushout, the diagram 
\[\begin{tikzcd}
    \Map_{\Alg(\Z)}(A\sslash x_{2n},B) \ar[r] \ar[d] & \Map_{\Alg(\Z)}(A,B) \ar[d] \\
    \ast\simeq \Map_{\Alg(\Z)}(\Z,B) \ar[r] & \Map_{\Alg(\Z)}(\Z[X_{2n}],B) \simeq \Omega^{\infty+2n}B
\end{tikzcd}\]
is a pullback. The associated long exact sequence in homotopy groups then implies all the claims.
\end{proof}

\begin{coro}\label{coro s2k maps are unique up to homotopy}
Let $B$ be a DGA whose homology is concentrated in even degrees. Then $\Map_{\Alg(\Z)}(S^m_{2n},B)$ is either empty or has homotopy concentrated in odd degrees. In particular,
if there exists a map of DGAs $S_{2n}^m\to B$, then it is unique up to homotopy.
\end{coro}
\begin{proof}
This follows from Lemma~\ref{lemm mapping spaces for pushout dgas} by induction over $n$, where we set $S_0^m=\Z$ and $x_0 =m$, since $S_{2n+2}^m = S^m_{2n}\sslash x_{2n}$.
\end{proof}

Since  $\colim_{n} S_{2n}^m \simeq \z/m$, the tower of DGAs given by $S_{2n}^m$ is an odd cell decomposition of $\z/m$ in the $\infty$-category of DGAs. As a result, we obtain the following.
\begin{coro}\label{prop map of dgas fp to a}
    Let $B$ be a DGA whose homology is concentrated in even degrees. If there exists a map $\z/m \to B$, then it is unique up to homotopy. 
\end{coro}
\begin{proof}
   
    If there is a map $\z/m \to B$, then it provides maps $S_{2n}^m \to B$ by precomposition. Commuting colimits, we have 
    \[\Map_{\Alg(\Z)}(\z/m , B) \simeq \on{lim}_{n} \Map_{\Alg(\Z)}(S_{2n}^m , B). \]
    All the spaces in the limit above are connected by Corollary \ref{coro s2k maps are unique up to homotopy}. Furthermore, the relevant $\lim^1$ term vanishes again due to (the second statement of) Lemma \ref{lemm mapping spaces for pushout dgas} applied to  Construction \ref{cons of the dgas s2km}.
\end{proof}

As an immediate consequence, we obtain the following non-formality result:
\begin{coro}\label{coro:truncations-of-S-not-formal}
   Let  $1 \leq k \leq \infty$, then $\tau_{\leq 2nk}S_{2n}^m$ is not formal.
\end{coro}
\begin{proof}
    If $\tau_{\leq 2nk}S_{2n}^m$ were formal, then there would be a DGA map $S_{2n}^m \to \Z/m \to \tau_{\leq 2nk}(S_{2n}^m)$ which differs from the truncation map $S_{2n}^m \to \tau_{\leq 2nk} S_{2n}^m$, contradicting the uniqueness of such maps.
\end{proof}

\subsection{A formality criterion for DGAs with polynomial homology}\label{sec:3.3}
In this section, we will give a sufficient condition for a DGA with polynomial homology to be formal. It will be based on exploiting the notion of centralizers of maps of $\bE_1$-ring spectra \`a la Lurie \cite[\S 5.3]{lurie2016higher}, which we will also use again later. We briefly review the relevant notions here.
\newline

To set the stage, let $R$ be a base $\bE_\infty$-ring. For a map of $\bE_k$-$R$-algebras $f \co A \to B$, Lurie constructs what is called the $\bE_k$-\emph{centralizer} of $f$, denoted by $\mathfrak{Z}^R(f)$, see \cite[Theorem 5.3.1.30]{lurie2016higher}. It is the terminal $\bE_k$-$R$-algebra fitting into the following commuting diagram of $\bE_k$-$R$-algebras.
\begin{equation}\label{eq definition of center of a map}
    \begin{tikzcd}
     & \Ztr^R(f) \otimes_R A \ar[dr]& \\
     A \ar[ur," u \otimes_R id_A"] \ar[rr,"f"] & & B
    \end{tikzcd}
\end{equation}
Here, $u\co R \to {\mathfrak{Z}^R(f)} $ is the unit map of the centralizer. The \emph{center} of an  $\bE_k$-$R$-algebra $B$ is by definition the centralizer of $\id_B$, we write $\Ztr^R(B)$ instead of $\Ztr^R(\id_B)$. It is naturally an $\bE_{k+1}$-$R$-algebra and $B$ is naturally an $\bE_{k}$-$\Ztr^R(B)$-algebra. When $A$ is the $\bE_k$-$R$-algebra underlying an $\bE_{k+1}$-$R$-algebra, extensions of the $\bE_k$-$R$-algebra structure on $B$ to an $\bE_k$-$A$-algebra structure are equivalently described as $\bE_{k+1}$-$R$-algebra maps $A \to \Ztr^R(B)$.

In this paper, we will only consider centralizers when $k=1$ in which case, the underlying $R$-module of $\Ztr^R(f)$ is given by the $R$-based topological Hochschild \emph{cohomology} spectrum
\[\mathfrak{Z}^R(f) \simeq \thh_R(A,B) := \map_{A\otimes_R A^{op}}(A,B). \]
Following the usual homological vs cohomological indexing conventions, we write $\thh^t_R(A,B)$ for $\pi_{-t}\thh_R(A,B)$.
\begin{rema}\label{Rmk:cohomology-dual-of-homology}
    For $R$ an $\bE_\infty$-ring and $f\colon A \to B$ a map of $R$-algebras with $B$ an $\bE_2$-$R$-algebra, there is in particular a canonical map of $R$-algebras $A \otimes_R A^{\op} \to B$. From the definition of $\thh^R(A,B)$, the $R$-based topological Hochschild \emph{homology} spectrum, we obtain the following equivalence:
    \begin{align*} 
    \thh_R(A,B) & = \Mor_{A \otimes_R A^{\op}}(A,B) \\
        & \simeq \Mor_B(A \otimes_{A\otimes_R A^{\op}} B,B) \\ 
        & =  \Mor_B(\thh^R(A,B),B)
         =: \thh^R(A,B)^{\vee_B}.
    \end{align*}
\end{rema}

 In what follows, we say that a spectrum is even if its odd homotopy groups vanish.
\begin{lemm}\label{lemma:THH}
    Assume $R$ and $A$ are connective and that $B$ is bounded below and even. Then
    \begin{enumerate}
        \item If $\thh_R(A,\pi_t(B))$ is even for all $t \in \Z$, then $\thh_R(A,B)$ is even.
        \item If furthermore the canonical map $\thh_R^0(A,\pi_t(B)) \to \pi_t(B)$ is surjective for all $t \in \Z$, then the canonical map $\thh_R^{-t}(A,B) \to \pi_t(B)$ is surjective for all $t\in \Z$.
    \end{enumerate}
\end{lemm}
\begin{proof}
    Since $A \otimes_R A^{\op}$ is connective, its category of modules comes with the usual Postnikov $t$-structure with truncation functors $\tau_{\leq k}$. In particular, $\pi_t(B)$ is indeed an $A$-bimodule and we have $\thh_R(A,B) \simeq \lim_{t} \thh_R(A,\tau_{\leq 2t}(B))$. Since $\thh_R(A,-)$ is an exact functor on $A$-bimodules, for each $t$ we find a fibre sequence
    \[ \thh_R(A,\pi_{2t}(B))[2t] \to \thh_R(A,\tau_{\leq 2t}(B)) \to \thh_R(A,\tau_{\leq 2t-2}(B))\]
    from which, together with assumption (1), we inductively deduce that the middle term is even for all $t$ and the latter map induces a surjection on even homotopy groups. It then follows form Milnor's $\lim$-$\lim^1$-sequence that $\thh_R(A,B)$ is even, and that $\thh_R(A,B) \to \thh_R(A,\tau_{\leq 2t}(B))$ is surjective on homotopy groups for all $t$. Therefore, to see the second statement, it suffices to show that the maps 
    \[ \thh_R(A,\tau_{\leq 2t}(B)) \to \tau_{\leq 2t}(B)\]
    are surjective on homotopy groups, which follows by the same filtration argument, making use of assumption (2) and the snake lemma.
\end{proof}
We will also use the following variant of Lemma~\ref{lemma:THH}

\begin{lemm}\label{lemm centralizer for the odd degree top formality}
     Assume that $R$, $A$, and $B$ are connective and that $\pis B$ is concentrated in degrees divisible by some $n>0$. Assume further that for every $t$:
     \[\thh_R(A,\pi_t(B)) \simeq \tau_{[-n,0]} \thh_R(A,\pi_t(B)).\] 
     If the canonical map $\thh_R^0(A,\pi_t(B)) \to \pi_t(B)$ is surjective for all $t \in \Z$, then the canonical map $\thh_R^{-t}(A,B) \to \pi_t(B)$ is also surjective for all $t\in \Z$.
\end{lemm}

\begin{proof}
    We  do induction on $s$ using the fiber sequence: 
    \[ \thh_R(A,\pi_{ ns}(B))[ns] \to \thh_R(A,\tau_{\leq ns}(B)) \to \thh_R(A,\tau_{\leq n(s-1)}(B)),\]
    to   prove  that the second map above is surjective in homotopy and that 
    \[\thh_R(A,\tau_{\leq ns}(B)) \to \tau_{\leq ns}(B)\]
    is also surjective in homotopy. For  $s=0$, the first statement follows since the right hand term is trivial and the second statement follows by hypothesis. For the inductive step, the first statement follows by the fact that $\thh_R(A,\tau_{\leq n(s-1)}B)$ is bounded above in homotopy degree $n(s-1)$ (by the Ext spectral sequence) and the hypothesis on the boundedness of the left hand term.  The second statement follows by the first statement, the induction hypothesis and the last hypothesis of the lemma.

    From this, the result follows by noting that $\thh_R(A,-)$ commute with limits and by considering Milnor's lim-$\textup{lim}^1$ sequence.
    
\end{proof}

\begin{rema}\label{rema grdd cmmtativ is suffcnt}
    A sufficient condition for the assumption in (2) in Lemma~\ref{lemma:THH} (or equivalently the last assumption in Lemma \ref{lemm centralizer for the odd degree top formality}) to hold is that $\pi_0B$ lies in the center of $\pis B$, and that $R$ and $A$ are connective. Indeed, we need to argue that every element $x$ in $\pi_t(B)$ can be represented as the image of $1 \in \pi_0(A)$ of an $A$-bimodule map $A \to \pi_t(B)$. 
    The composite
    \[ A \to \pi_0(A) \to \pi_0(B) \xrightarrow{\cdot x} \pi_t(B)\]
    then does the job since the last map is a map of $\pi_0B$-bimodules which forgets to an $A$-bimodule map through the composite of the first two $\bE_1$-$R$-algebra maps. 
\end{rema}

In what follows, $A[X_t]$ denotes $R[X_t] \otimes_{R} A$ where $R[X_t]$ is the free $R$-algebra on the $R$-module $\Sigma^tR$ as before.

\begin{prop}\label{prop:maps-to-centralizer}
    Assume that $R$ and $A$ are connective. If 
    \begin{enumerate}
        \item $B$ is bounded below, even, $\pis(B)$ is graded commutative, and that $\thh_R(A,\pi_t(B))$ is even for all $t$, or
        \item $B$ is connective, $\pis(B)$ is concentrated in degrees divisible by some $n>0$, $\pi_0B$ lies in the center of $\pi_*B$ and for all $t$ we have $\thh_R(A,\pi_t(B)) \simeq \tau_{[-n,0]} \thh_R(A,\pi_t(B))$,
    \end{enumerate} 
    then for all $x \in \pi_t(B)$, there exists a map $A[X_t] \to B$ in $\Alg(R)_{A/}$ sending $X_{t}$ to $x$.
\end{prop}    
\begin{proof}
    By Lemma~\ref{lemma:THH} and Remark \ref{rema grdd cmmtativ is suffcnt} and Lemma~\ref{lemm centralizer for the odd degree top formality} either of the assumptions (1) and (2) imply that the map $\thh_R(A,B) \to B$ is surjective on homotopy groups. Pick a lift $\bar{x} \in \thh_R^{-t}(A,B)$ of $x$ and consider the induced map 
    \[ R[X_t] \to \thh_R(A,B) = \Ztr^R(f).\]
    Then the canonical composite
    \[ A[X_t] = R[X_t] \otimes_R A \to \Ztr^R(f)\otimes_R A \to B\]
    is the desired map.
\end{proof}

\begin{rema}\label{Rmk:choices}
    We emphasize that the map $A[X_t] \to B$ of Proposition~\ref{prop:maps-to-centralizer} may not be unique. 
\end{rema}

Since we will use the following (well-known) fact several times, we record it here as a lemma. 
\begin{lemm}\label{lemma:THH-even}
    $\thh_\Z(\Z/m,\Z/m)$ is equivalent to $\Mor(\sph[\CP^\infty],\Z/m)$. In particular, it is even.
\end{lemm}
\begin{proof}
    First, we note that $\Z/m \otimes_\Z \Z/m = \Z/m[S^1]$ as $\bE_1$-algebras since both are the truncations of the free $\bE_1$-$\Z/m$-algebra on a generator of degree $1$. Under the equivalence of categories $\Mod(\Z/m[S^1]) \simeq \Fun(\CP^\infty,\Mod(\Z/m))$ the module $\Z/m$ corresponds to $r^*(\Z/m)$ where $r\colon \CP^\infty \to \ast$ is the unique map. Consequently, we find
    \[ \thh_\Z(\Z/m,\Z/m) \simeq \map_{\Fun(\CP^\infty,\Mod(\Z/m))}(r^*(\Z/m),r^*(\Z/m)) \simeq \Mor(\sph[\CP^\infty],\Z/m) \]
    as claimed.
\end{proof}
We finally obtain our main formality criterion for DGAs with polynomial homology.
\begin{coro}\label{cor:map-implies-formal}
    Let $B$ be a DGA with $\pis(B) \cong \Z/m[x_{2n}]/x_{2n}^k$ for some $1 \leq k \leq \infty$ and $n>0$. 
    \begin{enumerate}
        \item If there exists a map $\Z/m \to B$ of $\Z$-algebras, then $B$ is formal (under $\Z/m$).
        \item If $m=p$ is a prime and there exists a map $\fp \to B$ of $\sph$-algebras, then $B$ is topologically formal (under $\fp$).
        \item If $m=p$ is a prime and there exists a map $\fp \to B$ of $\MU$-algebras, then $B$ is formal as an $\MU$-algebra (under $\fp$).
    \end{enumerate}
\end{coro}
\begin{proof}
    We begin with (1). We wish to apply Proposition~\ref{prop:maps-to-centralizer} to $R=\Z$ and the map $\Z/m \to B$ to obtain a map $\Z/m[X_{2n}] \to B$ (under $\Z/m$) sending $X_{2n}$ to $x_{2n}$, which therefore induces an equivalence after applying $\tau_{\leq 2n(k-1)}(-)$. We then need to show that $\thh_\Z(\Z/m, \Z/m)$ is even which is the content of Lemma~\ref{lemma:THH-even}.

    To prove (2) and (3), by the same argument, it suffices to show that $\thh_\sph(\fp,\fp)$ and $\thh_\MU(\fp,\fp)$ are even. By Remark~\ref{Rmk:cohomology-dual-of-homology} it suffices to show that $\thh^\sph(\fp)$ and $\thh^\MU(\fp)$ are even. For the former, this is B\"okstedts classical computation, and for the latter see \cite[Theorem 10.2]{Lazarev} or \cite[Remark 2.4.3]{hahn2020redshift}.
\end{proof}

\begin{rema}\label{rema generalizing to witt vectors}
Let $k$ be a perfect field of characteristic $p$. The first part of the corollary above generalizes to $W(k)$-DGAs with homology $k[x_{2n}]/x_{2n}^k$ where $1 \leq k\leq \infty$ and $n>0$ as before. Namely, a $W(k)$-DGA with homology $k[x_{2n}]/x_{2n}^k$  is formal if it receives a $W(k)$-DGA map from $k$. The proof follows in the same way by using that $k \otimes_{W(k)} k \simeq k[S^1]$ so that by the same argument as in Lemma~\ref{lemma:THH-even}, $\thh_{W(k)}^*(k,k)\cong k[x_2]$ is even.     
\end{rema}

\subsection{Detecting non-formality}\label{sect nonformality via maps from s guys}
In this section, let us fix a prime number $p$ and let us write $S_{2n}$ for $S_{2n}^p$. We aim to determine for DGAs $A$ the smallest number $l$ such that $\tau_{\leq 2l}(A)$ is formal in terms of maps from suitable $S_{2n}^m$'s to $A$.

To begin, we recall  a result of Dugger and Shipley on the classification of DGAs with homology $\Lambda_{\fp}[x_{s}]$ for $s>0$ \cite[Example 3.15]{dugger2007topological}.
Indeed, such DGAs are  type $(\z/p,s)$-Postnikov extension of $\z/p$ in the $\infty$-category of DGAs in the sense of \cite[Section 1.2]{dugger2006postnikov}. These extensions are classified by the quotient of the Hochschild cohomology group $\hh_{\z}^{s+2}(\fp,\fp)$ by the action of  $\textup{Aut}(\fp)$ \cite[Proposition 1.5]{dugger2006postnikov}. As $\hh^{*}_{\z}(\fp,\fp) \cong \fp[v_2]$, one obtains that there are two such DGAs for even $s$ and there is a unique such DGA when $s$ is odd. The same applies to the classification of ring spectra with homotopy ring $\Lambda_{\fp}[x_{s}]$ as the relevant topological Hochschild cohomology groups $\thh_{\sph}^{*}(\fp,\fp)$ are given by the $\fp$-dual of $\thh_*^{\sph}(\fp,\fp) \cong \fp[u_2]$ (Remark \ref{Rmk:cohomology-dual-of-homology}). One obtains that there is a unique ring spectrum with homotopy $\Lambda_{\fp}[x_{s}]$ for odd $s$ and there are two for even $s$. 

It was observed in \cite{dugger2007topological} that the underlying ring spectrum of the non-formal DGA with homology $\Lambda_{\fp}[x_{2n}]$ is equivalent to the underlying ring spectrum of the formal one if and only if $2n\geq2p-2$. To see this, one looks at the map
\[\hh^{*}_{\z}(\fp,\fp) \to \thh^{*}_{\sph}(\fp,\fp)\]
which is the $\fp$-dual of the ring map 
\[\fp[u_2]\to \Gamma_{\fp}[u_2]\]
given by $\thh^{\sph}_*(\fp)\to \hh^{\z}_*(\fp)$ that sends $u_2$ to $u_2$. This map is an isomorphism for $*<2p $ and trivial for $*\geq 2p$ as desired.

\begin{term}\label{T_2n}
    For $n>0$, we denote by $T_{2n}$ the (unique) non-formal DGA with homology $\Lambda_{\fp}[x_{2n}]$.
\end{term}

By Corollary~\ref{coro:truncations-of-S-not-formal}, we have $\tau_{\leq 2n}S_{2n}^p \simeq T_{2n}$.
To generalize  to DGAs with homology $\Lambda_{\z/p^l}[x_{s}]$, we argue as in \cite[Example 3.15]{dugger2007topological}. The equivalence classes of such DGAs are given by the set  
\[\hh^{s+2}_{\z}(\z/p^l,\z/p^l)/\textup{Aut}(\Z/p^l).\]
By Lemma~\ref{lemma:THH-even} we have
\[\hh^{s+2}_{\z}(\z/p^l,\z/p^l) \cong \begin{cases}
    \Z/p^l & \text{ when $s$ is even } \\ 0 & \text{ when $s$ is odd.}
\end{cases} \]

Since the orbits of   $\z/p^l$ under the action of its units is given by the set with $l+1$ elements $\{[0],[p^0],[p^1],\dots,[p^{l-1}]\}$ and since $[0]$ provides the formal DGA, we obtain the first statement in the following. The second statement is also a consequence of Lemma~\ref{lemma:THH-even}.

\begin{lemm}\label{coro classification of DGAs with exterior homology}
Let $n>0$. The set of quasi-isomorphism classes of non-formal DGAs with homology $\Lambda_{\z/p^l}[x_{2n}]$ comes with a preferred bijection to $\{[0],[p^0],\dots,[p^{l-1}]\}$. Every DGA with homology $\Lambda_{\z/m}[x_{2n-1}]$ is formal.   
\end{lemm}
The following lemma will not be used in the rest of this work but we prove it here for the sake of completeness.

\begin{lemm}
    Under the bijection constructed above, we have $\tau_{\leq 2n}S_{2n}^{p^l}$ corresponds to $[p^0] \in \{[0],[p^0],\dots,[p^{l-1}]\}$. 
\end{lemm}
\begin{proof}
     For $i\geq 0$, let $C_i$ denote the DGA corresponding to $[p^i]$ above. By inspection on the pullback squares defining these DGAs, one obtains maps 
     \[f_{i}\co C_{i-1} \to C_i\]
     sending $x_{2n}$ to $px_{2n}$ where $i\leq l-1$ by using the maps $\z/p^n \xrightarrow{\cdot p} \z/p^n$ that carry a derivation corresponding to $[p^i]$ to a derivation corresponding to $[p^{i+1}]$.

     Assume to the contrary that $\tau_{\leq 2n} S_{2n}^{p^l}\simeq C_j$ for some $j \neq 0$. By Theorem~\ref{theo maps from snk to dgas with polynomial homology} below, there is a map $S_{2n}^{p^l} \to C_0$ that carries $x_{2n}$ to a non-trivial element. Then the composite 
     \[S_{2n}^{p^l} \to C_0 \to C_j \simeq \tau_{\leq 2n}S_{2n}^{p^j}\]
     does not agree with the Postnikov section $S_{2n}^{p^l} \to \tau_{\leq 2n}S_{2n}^{p^j}$ since $C_0 \to C_j$ carries $x_{2n}$ to $p^j x_{2n}$. This contradicts the uniqueness of such maps, Corollary \ref{coro s2k maps are unique up to homotopy}.
\end{proof}

An essential component of our methods is our identification of formality through maps out of the DGAs $\snm$.

\begin{theo}\label{theo maps from snk to dgas with polynomial homology}
    Let $A$ be a DGA with homology $\z/m[x_{2n}]/x_{2n}^{k}$ and $l \geq 1$. Then there is a DGA-map $S_{2l}^m \to A$ carrying $x_{2l}\in \pis S_{2l}^m$ to a non-trivial element in $\pis A$ if and only if $l$ is the smallest integer for which $\tau_{\leq 2l}A$ is not formal.
\end{theo}
\begin{proof}
    Let $l$ be the smallest integer for which $\tau_{\leq 2l}A$ is not formal and let $f\colon \Z/m \to \tau_{\leq 2(l-1)}A$ be the unit map of the formal DGA $\tau_{\leq 2(l-1)}A$. 
    First note that there is no map $S_{2s}^m \to A$ for $s>l$: Indeed, if there were, applying $\tau_{\leq 2l}(-)$ would result in a DGA map $\Z/m \to \tau_{\leq 2l}(A)$, which contradicts that $\tau_{\leq 2l}(A)$ is not formal by Corollary~\ref{cor:map-implies-formal}.
    
    Now, since $m=0$ in $\pi_0A$, there is a (unique) map of DGAs
    $S_{2}^m =\z \sslash m \to A$.
    Then we study extensions of this map to through the sequence
    \[\z \sslash m =S_2^m \to S_4^m \to S_6^m \to \cdots \to \colim_n S_{2n}^m \simeq \Z/m\]
    Let $1 \leq s\leq l-1$ and $g \co S_{2s}^m \to A$ be a DGA map. Then the two composites
    \[S_{2s}^m \xrightarrow{g} A \xrightarrow{\tau_{\leq 2(l-1)}} \tau_{\leq 2(l-1)}A \quad \text{ and } \quad S_{2s}^m \xrightarrow{\tau_{\leq0}} \Z/m \xrightarrow{f} \tau_{\leq 2(l-1)}(A)\]
    agree, by the uniqueness of such maps, Corollary~\ref{coro s2k maps are unique up to homotopy}. Hence $g$ induces the zero map on positive homotopy groups.
    Inductively, we deduce that for $s<l$ there is a (unique) map $S_{2s}^m \to A$ and that this map carries the generator $x_{2s} \in \pis S_{2s}^m$ to zero. Therefore, there is a (unique) DGA map $S_{2l}^m \to A$. By Remark \ref{remark:extensions-to-S2n}, This map is non-trivial on $\pi_{2l}$ as we have already observed that there is no DGA map $S_{2l+2}^m \to A$. 
    
    Conversely, assume that we have a DGA map $S_{2l}^m \to A$ that carries $x_{2l}$ to a non-trivial element. Then we obtain the map
    \[\z/m \simeq \tau_{\leq 2(l-1)}(S_{2l}^m) \to \tau_{\leq 2(l-1)}(A)\]
    which again implies that $\tau_{\leq 2(l-1)}(A)$ is formal. It remains to show that $\tau_{\leq 2l}(A)$ is not formal. Assume to the contrary that it is formal so that there is a map $\Z/m \to \tau_{\leq 2l}(A)$. Then by the uniqueness of DGA maps $S_{2s}^m \to \tau_{\leq 2l}(A)$, we deduce that the two composites
    \[S_{2l}^m \to A \xrightarrow{\tau_{\leq 2l}} \tau_{\leq 2l}A \quad \text{ and } \quad S_{2l}^m \to A \xrightarrow{\tau_{\leq 0}} \z/m \to \tau_{\leq 2l} A\]
    agree, contradicting the fact that $x_{2l}$ is mapped to a non-trivial element.
\end{proof}

\begin{theo}\label{thm:main-theorem}
For $1< k\leq \infty$,  $\tau_{\leq 2n(k-1)}S_{2n}^m$ is the unique DGA with homology $\Z/m[x_{2n}]/x_{2n}^k$ whose $2n$-Postnikov truncation is equivalent to $\tau_{\leq 2n}(S_{2n}^m)$. 
\end{theo}
\begin{proof}
By construction, $\tau_{\leq 2n(k-1)}S_{2n}^m$ is a DGA with homology $\z/m[x_{2n}]/x_{2n}^k$ and by Corollary~\ref{coro:truncations-of-S-not-formal}, it is not formal. Therefore, given another such DGA $A$, we need to show that $A \simeq \tau_{\leq 2n(k-1)} S_{2n}^m$.
By Theorem \ref{theo maps from snk to dgas with polynomial homology}, there is a (unique) map of DGAs $f \colon S_{2n}^m\to A$
 and this map carries $x_{2n} \in \pis \snm$ to a non-trivial element. Then the composite
 \[ S_{2n}^m \xrightarrow{f} A \to \tau_{\leq 2n}(A) \simeq \tau_{\leq 2n}(S_{2n}^m)\]
 is the canonical truncation map, again by uniqueness. Since the latter two maps induce isomorphisms on $\pi_{2n}$, so does the first. From the ring structure on the homotopy groups, we deduce that the induced map $\tau_{\leq 2n(k-1)}(S_{2n}^m) \to A$ is an equivalence as desired.
\end{proof}

\begin{coro}
    [Theorem \ref{ThmD}]\label{theo later uniquely defined nonformal dgas}
    For $1<k\leq \infty$, $\tau_{\leq 2n(k-1)}(S_{2n}^p)$ is the unique non-formal DGA with homology $\fp[x_{2n}]/x_{2n}^{k}$ having non-formal $2n$-Postnikov section. 
\end{coro}
\begin{proof}
    This follows from Theorem~\ref{thm:main-theorem} since there is a unique non-formal DGA with homology $\Lambda_{\fp}[x_{2n}]$ and $\tau_{\leq 2n}(S_{2n}^p)$ is not formal.
\end{proof}

\subsection{Root adjunctions}\label{sec adj roots} In this section, we aim to adjoin roots to the polynomial generators in $\pis(S_{2n}^m)$. To do so, we will need to study the Hochschild homology of $S_{2n}^m$.

We begin with the following which is immediate from  \cite[Lemma 2.2]{angeltveit2010thhoflandko} and \cite[Remark 4.6.3.16]{lurie2016higher}. 

 \begin{lemm}\label{prop hh over z with coefficients is tor dual}
     Let $R$ be a connective $\bE_\infty$-ring spectrum and let $A \to B$ be a map of $R$-algebras with $B$ an $\bE_2$-$R$-algebra. Then there is a canonical equivalence:
     \[\thh^{R}(A,B) \simeq B \otimes_{B \otimes_{R} A } B. \]
 \end{lemm}

 \begin{lemm}\label{lemm hh cohomology of snm}
  There are isomorphisms of graded abelian groups:
\begin{align*}
\hh^{\z}_*(\snm,\z/m) & \cong \z/m[u_2]/u_2^{n+1} \quad \text{ and } \\ 
\hh_\Z^*(\snm,\z/m) & \cong \z/m[w_{2}]/w_{2}^{n+1}.
\end{align*}
\end{lemm} 
 \begin{proof}
     By Lemma~\ref{prop hh over z with coefficients is tor dual} and Construction~\ref{cons of the dgas s2km}, we have
     \[ \hh^\Z(S_{2n}^m,\Z/m) \simeq \Z/m \otimes_{\Z/m\otimes_\Z S_{2n}^m} \Z/m \simeq \Z/m \otimes_{\Z/m[\Omega \CP^n]} \Z/m \simeq \Z/m[\CP^n]\]
     so Remark~\ref{Rmk:cohomology-dual-of-homology} gives 
     \[\hh_\Z(S_{2n}^m,\Z/m) \simeq \map_{\Z/m}(\Z/m[\CP^n],\Z/m) \simeq \Mor(\sph[\CP^n],\Z/m).\]
     The claims then follow from the computations of the (co)homology of $\CP^n$.
 \end{proof}

We now move towards the proof of Theorem \ref{ThmC} (1), i.e.\ we construct infinitely many non-equivalent DGAs with homology $\z/m[x_{2n}]$. These DGAs are constructed by adjoining roots to the DGAs $S_{2n}^m$ (as in \cite[Construction 4.6]{ausoni2022adjroot}).

\begin{prop}\label{prop snm is a zxn algebra}
    The DGA $\snm$ admits the structure of a $\z[X_{2n}]$-algebra where $X_{2n}$ acts through $x_{2n}\in \pis \snm$.
\end{prop}
\begin{proof}
    By Lemma~\ref{lemma:THH}, Remark \ref{rema grdd cmmtativ is suffcnt} and Lemma~\ref{lemm hh cohomology of snm} we find that \[\Ztr^\Z(S_{2n}^m) = \thh_\Z(S_{2n}^m,S_{2n}^m)\] is even and that the map $\Ztr^\Z(S_{2n}^m) \to S_{2n}^m$ is surjective on homotopy groups. Choose a lift $\bar{x}_{2n}$ of $x_{2n}$ and consider the associated $\bE_1$-$\Z$-algebra map \[\Z[X_{2n}] \to \Ztr^\Z(S_{2n}^m).\] Since its target is even, it follows from \cite[Proposition 3.15]{ausoni2022adjroot} that this map extends to an $\bE_2$-$\Z$-algebra map, making $S_{2n}^m$ into the desired $\Z[X_{2n}]$-algebra.
\end{proof}

\begin{rema}\label{remark:non-unqieness}
    The $\Z[X_{2n}]$-algebra structure on $S_{2n}^m$ is not canonical. In fact, in the above argument we have made two choices: that of a lift $\bar{x}_{2n}$ of $x_{2n}$ and that of an extension of the resulting $\bE_1$-map $\Z[X_{2n}] \to \Ztr^\Z(S_{2n}^m)$ to an $\bE_2$-map.

    Nevertheless, let us now fix a choice of a $\Z[X_{2n}]$-algebra structure on $S_{2n}^m$ as in Proposition~\ref{prop snm is a zxn algebra}. We will always use this choice unless we state otherwise.
\end{rema}

We now note that there are canonical $\bE_\infty$-$\z$-algebra maps
\[\z[X_{2nl}] \to \z[X_{2n}]\]
that carry $X_{2nl}$ to $X_{2n}^l$ in homotopy, because both sides are formal as $\bE_\infty$-$\Z$-algebras. 
Through this map, $- \otimes_{\z[X_{2nl}]} \z[X_{2n}]$ defines a functor from the $\infty$-category of $\bE_k$-$\z[X_{2nl}]$-algebras to $\bE_k$-$\z[X_{2n}]$-algebras for any $k\geq 1$.

\begin{cons}\label{cons adj root to snm}
Let $l>0$ and recall that we have fixed a choice of $\Z[X_{2n}]$-algebra structure on $S_{2nl}^m$ (Remark~\ref{remark:non-unqieness}). Following \cite[Construction 4.6]{ausoni2022adjroot} we define: 
\[ S_{2nl}^m \big[\sqrt[l]{x_{2nl}}\big] := S_{2nl}^m \otimes_{\z[X_{2nl}]} \z[X_{2n}]\]
so that $S_{2nl}^m [\sqrt[l]{x_{2nl}}] $ is a DGA equipped with a map of DGAs
\[f\co S_{2nl}^m \to S_{2nl}^m [\sqrt[l]{x_{2nl}}]\]
and the Tor spectral sequence shows that there is an isomorphism of graded rings 
\[\pis (S_{2nl}^m [\sqrt[l]{x_{2nl}}]) \cong \z/m[x_{2n}],\]
i.e.\ this construction adjoins an $l$ root to $x_{2nl}$ at the level of homotopy groups. In particular, the map $f$ above carries $x_{2nl}$ to $x_{2n}^l$ on the level of homotopy groups. 
\end{cons}

\begin{rema}\label{rema root adj is not unique}
As explained in Remark~\ref{remark:non-unqieness}, we know of no preferred choice for a $\Z[X_{2nl}]$-algebra structure on $S_{2nl}^m$. We do not know in what way the $\bE_1$-$\Z$-algebra structure on $S_{2nl}^m[\sqrt[l]{x_{2nl}}]$ depends on such choices.
\end{rema}

\begin{coro}\label{coro truncations of root adj to snm}
  The DGA  $\tau_{\leq t }S_{2nl}^m[\sqrt[l]{x_{2n}}]$ is non-formal if and only if $t \geq 2nl$.
\end{coro}
\begin{proof}
    This follows by applying Theorem \ref{theo maps from snk to dgas with polynomial homology} to the map $S_{2nl}^m \to S_{2nl}^m[\sqrt[l]{x_{2n}}]$ from  Construction \ref{cons adj root to snm}. 
\end{proof}
\begin{coro}\label{coro root adjs are distinct}
    In the situation of Construction \ref{cons adj root to snm}, we have 
    \[S_{2nl}^m [\sqrt[l]{x_{2nl}}] \not \simeq S_{2nl'}^m [\sqrt[l']{x_{2nl'}}]\]
    as DGAs whenever $l \neq l'$.
\end{coro}
\begin{proof}
Assume $l'<l$, then the $2nl'$-Postnikov sections of these DGAs are not quasi-isomorphic by Corollary~\ref{coro truncations of root adj to snm}.
\end{proof}

We arrive at one of the main results of this paper (Theorem~\ref{ThmC} (1)), in particular that for all $n>0$ and $m>1$, there are  infinitely many quasi-isomorphism classes of DGAs with homology $\z/m[x_{2n}]$.

\section{Topological formality for DGAs with polynomial homology}\label{sect topological formality}

Here, our goal is to prove Theorem \ref{ThmB} and our other results on the topological formality/non-formality of DGAs with polynomial homology. 
\subsection{Topological formality for the $m=p$ case}\label{subsec topological formality for m is p}
We begin with proving that $S_{2p-2}:= S_{2p-2}^p$ is topologically formal, i.e.\ that it is equivalent to the formal DGA $\fp[x_{2p-2}]$ as a ring spectrum. Our approach will be built on analyzing $\pi_0(\Ztr^\Z(S_{2p-2}))$. To that end, we need some preliminary computations. Recall that we have $\tau_{\leq 2p-2}(S_{2p-2}) \simeq T_{2p-2}$ (Terminology~\ref{T_2n}) and that $T_{2p-2}$ is topologically formal, so that in particular, there is a ring spectrum map $\fp \to T_{2p-2}$.

\begin{lemm}\label{lemm technical for later hh of s to t}
We have the following connectivity estimates:
\begin{enumerate}
    \item\label{item:connectivity1} The map $\hh^\Z_*(S_{2p-2},\fp) \to \hh^\Z_*(T_{2p-2},\fp)$ is an isomorphism for $*< 4p-4$,
    \item\label{item:connectivity2} All maps in the following composite are isomorphisms for $* \leq 2p-2$
    \[ \thh_*(\fp) \to \thh_*(T_{2p-2},\fp) \to \hh^\Z_*(T_{2p-2},\fp) \to \hh^\Z(\fp,\fp).\]
\end{enumerate}
\end{lemm}
\begin{proof}
\eqref{item:connectivity1} follows from the fact that Hochschild homology preserves connectivity and that $S_{2p-2}\to T_{2p-2}$ is a $(4p-5)$-Postnikov truncation. 
For \eqref{item:connectivity2}, we begin by noting that the whole composite identifies with the map $\fp[u_2] \to \Gamma_{\fp}[u_2]$ which is an isomorphism for degrees $<2p$.
Next, we show that the last map
$\hh^{\z}_*(T_{2p-2},\fp) \to \hh^{\z}_*(\fp,\fp)$
is an isomorphism for $*<2p$. By \eqref{item:connectivity1}, we may replace $T_{2p-2}$ with $S_{2p-2}$, after which, using Lemma~\ref{prop hh over z with coefficients is tor dual}, the map in question becomes equivalent to the map 
\begin{equation}\label{eq:map}
\pis (\fp \otimes_{\fp \otimes_{\z} S_{2p-2}} \fp ) \to \pis (\fp \otimes_{\fp \otimes_{\z} \fp} \fp)
\end{equation}
induced by $\fp \otimes_\Z S_{2p-2} \to \fp \otimes_\Z \fp$. This map is a $1$-Postnikov truncation as $\fp \otimes_\Z -$ preserves connectivity (and the target is 1-truncated).
By Lemma~\ref{lemma:1-truncation-CP}, it agrees up to an autoequivalence of $\fp \otimes_{\z} \fp$ with the $\fp$-algebra map $\fp[\Omega\CP^{p-1}] \to \fp[S^1]$. Using Corollary \ref{coro bar const on cpn}, we deduce that the map \eqref{eq:map} is equivalent to the map $\pis \fp[\CP^{p-1}] \to \pis \fp[\CP^\infty]$ which is an isomorphism for $*<2p$ as desired. 

It will then suffice to show that the map $\thh_*(T_{2p-2},\fp) \to \thh_*(\fp,\fp)$ (which is a retraction of the first map in the above composite) is an isomorphism for $*<2p-1$. Again using Lemma~\ref{prop hh over z with coefficients is tor dual}, the map is equivalent to the map induced on Bar constructions from the map $\fp\otimes_\sph T_{2p-2} \to \fp \otimes_\sph \fp$. This map induces an isomorphism on $\pis$ for $*<2p-2$ and a surjection on $\pi_{2p-2}$ as the functor $- \otimes_\sph \fp$ preserves connectivity. Now, the Bar construction raises connectivity by 1, so the map under investigation induces an isomorphism for $* < 2p-1$ as needed.
\end{proof}

 Taking $\fp$-duals, we obtain the following (see Remark~\ref{Rmk:cohomology-dual-of-homology}). 

\begin{coro}\label{coro technical map of hh cohomology}
     The map $\hh_\Z^*(T_{2p-2},\fp) \to \hh_\Z^*(S_{2p-2},\fp)$ is an isomorphism for $*< 4p-4$ and the map \[ \hh^*_{\z}(T_{2p-2},\fp) \to \thh_\sph^*(T_{2p-2},\fp) \] is an isomorphism for $*\leq 2p-2$.
\end{coro}

\begin{lemm}\label{lemm center has p equals zero}
    We have $p=0$ in $\hh_\Z^0(S_{2p-2},S_{2p-2})$.
\end{lemm}
\begin{proof}
    We will prove that all of the following maps induce isomorphisms on $\pi_0$.
    \begin{enumerate}
        \item $\hh_\Z(S_{2p-2},S_{2p-2}) \to \hh_\Z(S_{2p-2},T_{2p-2})$
        \item $\hh_\Z(T_{2p-2},T_{2p-2}) \to \hh_\Z(S_{2p-2},T_{2p-2})$
        \item $\hh_\Z(T_{2p-2},T_{2p-2}) \to \thh_\sph(T_{2p-2},T_{2p-2})$
    \end{enumerate}
Once this is established, we use that $T_{2p-2}$ is topologically formal, and hence equivalent to an $\bE_\infty$-$\fp$-algebra. It follows that $\thh_\sph^0(T_{2p-2},T_{2p-2})$ is an $\fp$-algebra, and hence the lemma.

Now let us recall from Lemma~\ref{lemm hh cohomology of snm} the isomorphism of graded abelian groups
\begin{equation}\label{eq recall hh cohomology of s}
\hh_\Z^*(S_{2p-2},\fp)\cong \fp[w_2]/w_2^p.
\end{equation}

For (1), we consider the filtration:
\[\cdots \to \hh_{\z}( S_{2p-2},\tau_{\leq 4(2p-2)}S) \to \hh_{\z}( S_{2p-2},\tau_{\leq 2(2p-2)}S) \to\hh_{\Z}(S_{2p-2},T_{2p-2}),\]
obtained from the Postnikov tower of $S_{2p-2}$ whose limit is $\hh_{\z}(S_{2p-2},S_{2p-2})$. Due to \eqref{eq recall hh cohomology of s}, all the maps above are $\pi_0$ (and $\pi_1$) isomorphisms and the relevant $\textup{lim}^1$ term vanishes, giving (1).   

For (2), we apply the fiber sequence \[\fp[2p-2] \to T_{2p-2} \to \fp\] on the coefficients (of the Hochschild cohomology spectra in (2)) and consider the induced map of long exact sequences by (2). It follows by  Corollary~\ref{coro technical map of hh cohomology} and \eqref{eq recall hh cohomology of s} that the $\pi_0$ of the map in (2) sits in the middle of a short exact sequence with outer terms given by   \[\hh_\z^k(S_{2p-2},\fp) \to \hh_\Z^k(T_{2p-2},\fp)\] for $k=2p-2$ (as the kernel term) and $k=0$ (as the cokernel term) which are isomorphisms, giving (2).

Likewise, the final map sits in the middle of an exact sequence with outer terms \[\hh_\Z^k(T_{2p-2},\fp) \to \thh_\sph^k(T_{2p-2},\fp)\] again for $k=2p-2,0$. It follows by Corollary \ref{coro technical map of hh cohomology} and \eqref{eq recall hh cohomology of s} that these maps are isomorphisms and that the snake lemma applies to prove (3). 
\end{proof}

\begin{theo}\label{thm:topological-formality}
    For all $n\geq p-1$, the DGA $S_{2n}$ is topologically formal.
\end{theo}
\begin{proof}
 Since there are DGA maps $S_{2p-2} \to S_{2n}$ by construction, using Corollary~\ref{cor:map-implies-formal}, it suffices to construct an $\sph$-algebra map $\fp \to S_{2p-2}$. To that end, note that 
    by Lemma \ref{lemm center has p equals zero}, $p = 0$ in $\pi_0 \mathfrak{Z}^{\Z}(S_{2p-2})$. By the Hopkins-Mahowald theorem, see e.g.\ \cite[Theorem 5.1]{barthel}, $\fp$ is the free $\bE_2$-ring spectrum with $p=0$. We therefore obtain a map of $\bE_2$-ring spectra
    $\fp \to \mathfrak{Z}^{\z}(S_{2p-2})$
   which we can compose with the DGA map $\Ztr^\Z(S_{2p-2}) \to S_{2p-2}$.
\end{proof}

Let us point out that for $p=2$, this simply means that all $S_{2n}$ are topologically formal. For odd $p$, the same is not true: $S_2=\Z\sslash p$ is never topologically formal \cite{davis2023cyclic}. In fact, we also have the converse to Theorem~\ref{thm:topological-formality}. First, an observation:

\begin{observation}\label{ob:p-local-truncated-rings}
Let $p$ be a prime and $n<p-1$. Then the map $\sph \to \Z$ induces an equivalence \[\Alg(\sph)_{[0,2n]}^{(p)} \simeq \Alg(\Z)_{[0,2n]}^{(p)}\] between between $p$-local $\sph$-algebras and $p$-local $\Z$-algebras which are connective and $2n$-truncated. This follows simply from the fact that the map $\tau_{\leq 2n}(\sph)\to \Z$ is a $p$-local equivalence when $n<p-1$.
\end{observation}

\begin{prop}\label{prop:S_2n-not-formal}
    For $n<p-1$, the DGA $S_{2n}$ is not formal over $\sph$.
\end{prop}
\begin{proof}
    We have recorded already that $\tau_{\leq 2n}S_{2n}$ is not formal over $\Z$ (Corollary \ref{coro:truncations-of-S-not-formal}). The claim then follows from Observation~\ref{ob:p-local-truncated-rings}.
\end{proof}

As a consequence of the topological formality mentioned  above, we also deduce:

\begin{theo}\label{theo topological formality even degrees}
    Let $n> 0$ and $1<k\leq \infty$. Every DGA with homology $\fp[x_{2n}]/x_{2n}^k$  and a topologically formal $(2p-4)$-Postnikov section is topologically formal. 
\end{theo}
\begin{proof}
    Let $A$ be a DGA satisfying the hypothesis. If $A$ if formal as a DGA then it is topologically formal. Assume that $A$ is not formal as a DGA. By the equivalence $\Alg(\sph)_{[0,2p-4]}^{(p)} \simeq \Alg(\Z)_{[0,2p-4]}^{(p)}$ of Observation~\ref{ob:p-local-truncated-rings}, we deduce that $\tau_{\leq 2p-4} A$ is formal. Then there is a map of DGAs $S_{2l} \to A$ for some $2l>2p-4$ due to Theorem \ref{theo maps from snk to dgas with polynomial homology}. Since $S_{2l}$ is topologically formal due to Theorem \ref{thm:topological-formality}, there is a map of ring spectra $\fp \to S_{2l}$. Applying Corollary \ref{cor:map-implies-formal} to the composite $\fp \to S_{2l} \to A$ gives the desired result.  
\end{proof}

\begin{proof}[Proof of Theorem \ref{ThmA}]
Theorem \ref{ThmA} (1) and the first statement of Theorem \ref{ThmA} (2) is a consequence of Corollary \ref{cor:map-implies-formal}. The rest of the statements follow by Theorem \ref{theo topological formality even degrees}. 
\end{proof}

The following result also covers the case of odd degree generators. 
\begin{theo}[Theorem \ref{ThmB}]\label{theo topological formality odd degrees}
    Let $n \geq 2p-2$ and $1<k\leq \infty$, then every DGA with homology $\fp[x_{n}]/x_{n}^k$ is topologically formal. 
\end{theo}
\begin{proof}
    Let $B$ be a DGA as in the theorem, then there is a DGA map $S_{2p-2} \to B$ (Remark \ref{remark:extensions-to-S2n}). By Lemma \ref{lemm hh cohomology of snm}, we may apply  Proposition \ref{prop:maps-to-centralizer} to the DGA map $S_{2p-2}\to B$ to obtain a map $f \co S_{2p-2}\otimes_\Z \Z[X_n] \to B$ of DGAs where $X_n$ is mapped to $x_n$. Then the canonical composite, induced by a map of ring spectra $\fp \to S_{2p-2}$ which exists due to Theorem \ref{thm:topological-formality},
    \[\fp[X_n] \simeq \fp \otimes_{\sph} \sph[X_n] \to S_{2p-2} \otimes_{\sph} \sph[X_{n}] \xrightarrow{\simeq} S_{2p-2} \otimes_{\z} \z[X_n] \xrightarrow{f} B\]
    induces on homotopy groups the map $\fp[X_n] \to \fp[x_n]/x_n^k$, and therefore an equivalence upon applying $\tau_{\leq n(k-1)}$ as needed.
\end{proof}

Similarly, we have:
\begin{coro}\label{cor:root-adjunction-topologically-formal}
    Let $nl\geq p-1$. Then $S_{2nl}[\sqrt[l]{x_{2nl}}]$ is topologically formal.
\end{coro}
\begin{proof}
There are maps (Theorem~\ref{thm:topological-formality})
    \[ \fp \to S_{2nl} \to S_{2nl}[\sqrt[l]{x_{2nl}}]\]
    so that we may apply Corollary~\ref{cor:map-implies-formal} to deduce formality over $\sph$.
\end{proof}

\begin{proof}[Proof of Theorem \ref{ThmC}]
    Theorem \ref{ThmC} (1) follows by Corollary \ref{coro root adjs are distinct} and Theorem \ref{ThmC} (2) follows by Corollary \ref{cor:root-adjunction-topologically-formal} above. 
\end{proof}

\begin{rema}
    To the best of our knowledge, the above provides the first explicit examples of infinitely many non quasi-isomorphic $\z$-algebra structures on a single $\sph$-algebra. We will use this later, to construct infinitely many $\Z$-linear structures on a fixed stable $\infty$-category, and in particular to construct exotic dg-enhancements of certain triangulated categories, see Section~\ref{sec:exotic}.
\end{rema}

Conversely, we also find:
\begin{prop}\label{prop root adj guys that are topologically distinct for all m}
    For fixed $n$ and $l'<l< \frac{p-1}{n}$, the DGAs $S_{2nl}^m[\sqrt[l]{x_{2nl}}]$ and $S_{2nl'}^m[\sqrt[l']{x_{2nl}}]$ are not topologically  equivalent and neither are topologically formal.
\end{prop}
\begin{proof}
   By Corollary \ref{coro truncations of root adj to snm}, the two DGAs in question remain non-equivalent over $\Z$ after applying $\tau_{\leq 2nl'}$ and remain non-formal over $\Z$ after applying $\tau_{\leq 2nl}$. The result then follows from the canonical equivalence $\Alg(\sph)_{[0,2p-4]}^{(p)} \simeq \Alg(\Z)_{[0,2p-4]}^{(p)}$ from Observation \ref{ob:p-local-truncated-rings}.
\end{proof}

\subsection{Topological non-formality in the $m=p^s$ case}

For the rest of the section, let $s \geq 3$ for $p=2$ and let $s \geq 2$ for an odd prime $p$. This ensures that $\sph/p^s$ is an $\bE_1$-algebra (i.e.\ a ring spectrum) \cite{burklund2022multiplicative}. From this, we prove that topological equivalences agree with quasi-isomorphisms in many cases we considered so far. 

\begin{prop}\label{prop topological formality is formality for s geq 2}
    Let $n>0$, $1<k\leq \infty$ and let $s$ be as above. Then a DGA with homology $\z/p^s[x_{2n}]/x_{2n}^k$ is  formal if and only if it is topologically formal. 
\end{prop}
\begin{proof}
    Let $A$ be a DGA as above. One direction is immediate. Now assume that $A$ is topologically formal, then there is a composite map of ring spectra 
    \[\sph/p^s \to \z \otimes_{\sph} \sph/p^s \simeq \z/p^s \to A.\]
    By adjunction, one obtains a map of $\z$-algebras 
    \[\z \otimes_{\sph} \sph/p^s \simeq \z/p^s \to A\]
    which implies the formality of $A$ by Corollary \ref{cor:map-implies-formal}.
\end{proof}
The same proof gives the following, using that $\MU/m$ is an $\bE_1$-$\MU$-algebra for all $m$ \cite{angeltveit2008thhofainftyringspectra}.
\begin{coro}\label{coro mu formality is formality}
     Let $n>0$ and $1<k\leq \infty$.  Then a DGA with homology $\z/m[x_{2n}]/x_{2n}^k$ is formal if and only if it is formal as an $\MU$-algebra. 
\end{coro}

\begin{prop}\label{prop infty many dgas up to topological equiv}
    Let $s$ be as above and $n>0$, then for every $l \neq l'$
    \[S_{2nl}^{p^s}[\sqrt[l]{x_{2nl}}]\not \simeq  S_{2nl'}^{p^s}[\sqrt[l]{x_{2nl'}}]
    \]
    as $\sph$-algebras and 
    \[S_{2nl}^{m}[\sqrt[l]{x_{2nl}}]\not \simeq  S_{2nl'}^{m}[\sqrt[l]{x_{2nl'}}]\]
    as $\MU$-algebras. 
\end{prop}

\begin{proof}
    Assume $l'<l$, then  by Corollary \ref{coro truncations of root adj to snm}, the $2nl'$-truncation of $S_{2nl}^{p^s}[\sqrt[l]{x_{2nl}}]$ is formal but of $ S_{2nl'}^{p^s}[\sqrt[l]{x_{2nl'}}]$ is not formal. It follows by  Proposition \ref{prop topological formality is formality for s geq 2} that $2nl'$-truncation of $S_{2nl}^{p^s}[\sqrt[l]{x_{2nl}}]$ is \textit{topologically} formal but of $ S_{2nl'}^{p^s}[\sqrt[l]{x_{2nl'}}]$ is not \textit{topologically} formal. This proves the first statement. The second statement follows similarly by using Corollary \ref{coro mu formality is formality}. 
\end{proof}

\begin{coro}\label{coro infty many dgas up to topological equivalences}
    Let $n>0$ and $s$ above. Up to topological equivalence, there are infinitely many DGAs with homology $\z/p^s[x_{2n}]$. 
    
    Similarly, up to $\MU$-algebra equivalence, there are infinitely many DGAs with homology $\Z/m[x_{2n}]$.
\end{coro}
\begin{proof}
   Proposition \ref{prop infty many dgas up to topological equiv} provides infinite families of such DGAs.
\end{proof}

\section{DGAs with exterior homology}

In \cite[Theorem 1.1]{DGI} the authors classify DGAs whose homology is an exterior algebra over $\fp$ on a  generator in degree $-1$ in terms of CDVRs with residue field $\fp$. They finish the paper with the statement ``We do not know how to classify DGAs with exterior homology over $\fp$ in a degree $-n$ generator'' where $n>1$. To the best of our knowledge, they in fact did not know of any non-formal examples. 

However, in \cite[Proposition 6.1]{DGI} it was observed that for $n>1$ by Koszul duality, quasi-isomorphism classes of DGAs with exterior homology in a negative degree $-n$ generator (over a commutative ring $R$) correspond bijectively to quasi-isomorphism classes of DGAs with polynomial homology in a positive degree $n-1$ generator. Indeed, given a DGA $A$ with homology $\Z/m[x_{n-1}]$, one may consider the Koszul dual algebra $KD(A) =\Mor_A(\Z/m,\Z/m)$. This is then a $\Z$-algebra with homology $\Lambda_{\Z/m}[e_{-n}]$ and Koszul duality says that $A$ can be recovered as the Koszul dual of $KD(A)$, i.e.\ one has $A \simeq \Mor_{KD(A)}(\Z/m,\Z/m)$.

\begin{coro}\label{coro exterior inftly many dgas}
    Let $n<-1$ be odd. Up-to quasi-isomorphisms, there are  infinitely many  DGAs with homology $\Lambda_{\Z/m}[x_n]$.
\end{coro}
\begin{proof}
    By the above Koszul duality argument, this follows from Theorem~\ref{ThmC}.
\end{proof}

\begin{rema}
However, we find that the infinitely many quasi-isomorphism classes coming from the DGAs $S_{2nl}^m[\sqrt[l]{x_{2nl}}]$ collapse to only finitely many equivalence classes of ring spectra for $m=p$: Indeed, if $A \simeq B$ as ring spectra, then one obtains an induced equivalence of Koszul duals $KD(A) \simeq KD(B)$, so that we may appeal to Corollary~\ref{cor:root-adjunction-topologically-formal}.    
\end{rema}

\begin{coro}\label{coro later exterior dgas are topologically formal}
    Let $A$ be a DGA with homology $\Lambda_{\fp}[e]$ with $|e|<-(2p-2)$. Then $A$ is topologically formal. 
\end{coro}
\begin{proof}
    As we have just noted, if $A$ and $B$ are topologically equivalent, then so are their Koszul duals $KD(A)$ and $KD(B)$. By the endomorphism description of the Koszul dual, it is clear that the Koszul dual of an $\fp$-DGA is also an $\fp$-DGA. In particular, we have $KD(\Lambda_{\fp}[e]) = \fp[x]$ with $|x|=-|e|-1$, i.e.\ the Koszul dual of the formal DGA $\Lambda_{\fp}[e]$ is the formal DGA $\fp[x]$. 
    The statement of the Corollary is therefore equivalent to the statement that every DGA with homology $\fp[x]$ with $|x| = -|e|-1 > 2p-3$ is topologically formal, which is the statement of Theorem \ref{ThmB}.
\end{proof}

By the same arguments, we have the following corollaries of Corollary \ref{coro infty many dgas up to topological equivalences}. 
\begin{coro}\label{coro exterior inftly many dgas up to topological equiv}
    Let $n<-1$ be odd, $s\geq 3$ ($s\geq 2$ if $p$ is odd). Up to topological equivalence, there are infinitely many DGAs with homology $\Lambda_{\z/p^s}[x_{n}]$. 
    
    Similarly, up to $\MU$-algebra equivalence, there are infinitely many DGAs with homology $\Lambda_{\Z/m}[x_{n}]$.
\end{coro}

\section{Exotic DG-enhancements}\label{sec:exotic}
In this section, let us fix a prime $p$ and for ease of notation write $S_{2n} := S_{2n}^p$.
We recall that we have fixed a $\z[X_{2n}]$-algebra structure on $S_{2n}$ for each $n$ (Remark~\ref{remark:non-unqieness}) to define the root adjunctions $S_{2nl}[\sqrt[l]{x_{2nl}}]$. Let us then consider the following family of DGAs: 
\[F_{2n}(l) := S_{2nl}[\sqrt[l]{x_{2nl}}][x_{2nl}^{-1}]; \]
Note that $\pis(S_{2nl}[\sqrt[l]{x_{2nl}}]) \cong \fp[t_{2n}]$ is graded commutative, so such a localisation exists and satisfies 
$\pis(F_{2n}(l)) \cong \fp[t_{2n}^{\pm 1}]$.
\begin{prop}\label{prop:different-module-categories}
    The $\Z$-linear $\infty$-categories $\Mod(F_{2n}(l))$ and $\Mod(F_{2n}(l'))$ are $\Z$-linearly equivalent if and only if $l=l'$.
\end{prop}

\begin{proof}
    For the non-trivial implication, let $\Phi\colon \Mod(F_{2n}(l)) \to \Mod(F_{2n}(l'))$ be a $\Z$-linear equivalence. Since $\Phi$ is fully faithful, the induced map 
    \[ F_{2n}(l) = \mathrm{end}_{F_{2n}(l)}(F_{2n}(l))\to \mathrm{end}_{F_{2n}(l')}(\Phi(F_{2n}(l)))\]
    is an equivalence of DGAs. Since $\pis F_{2n}(l')$ is a field, we find that $\Phi(F_{2n}(l))$ 
    is a coproduct of shifted copies $F_{2n}(l')$ and since the map above is an equivalence, we deduce that it is just a shift of a single copy of $F_{2n}(l')$. In particular, the right hand side above is given by $F_{2n}(l')$; so the equivalence above is an equivalence of DGAs $F_{2n}(l) \simeq F_{2n}(l')$.
   
  Without loos of generality assume $l'\geq l$ and consider the canonical maps out of $S_{2nl}$:
    \[ S_{2nl} \to F_{2n}(l) \simeq F_{2n}(l') \leftarrow S_{2nl'}\leftarrow S_{2nl}\]
    where the left map carries $x_{2nl}$ to a non-trivial element. By Corollary~\ref{coro s2k maps are unique up to homotopy}, the same is true for the right composite, so that for degree reasons, we find $l'=l$ as claimed. In particular, for $l'\neq l$, we deduce that $\Mod(F_{2n}(l))$ and $\Mod(F_{2n}(l'))$ are not $\Z$-linearly equivalent.
\end{proof}

Note that in the following, $\fp[t_{2n}^{\pm 1}]$ denotes the formal DGA with homology $\fp[t_{2n}^{\pm 1}]$.

\begin{coro}[Corollary \ref{CorE}]
    The family \[\{\Mod(F_{2n}(l))\}_{l\geq \frac{p-1}{n}}\] consists of pairwise distinct $\Z$-linear $\infty$-categories that are all equivalent as stable $\infty$-categories to $\Mod(\fp[t_{2n}^{\pm 1}])$. In particular, the triangulated category $\textup{Ho}(\Mod(\fp[t_{2n}^{\pm}])$ admits infinitely many DG-enhancements.
\end{coro}
\begin{proof}
    That the described $\Z$-linear categories are pairwise inequivalent is the content of Proposition~\ref{prop:different-module-categories}. To see that the underlying stable $\infty$-categories are all equivalent we recall from Corollary~\ref{cor:root-adjunction-topologically-formal} that each $S_{2nl}[\sqrt[l]{x_{2nl}}]$ is topologically equivalent to the formal DGA $ \fp[t_{2n}]$ whenever $nl\geq p-1$ and therefore each $F_{2n}(l)$ above is topologically equivalent to the formal DGA $\fp[t_{2n}^{\pm 1}]$. In particular,  $\Mod(F_{2n}(l)) \simeq \Mod(\fp[t_{2n}^{\pm 1}])$ is independent of $l$ as stable $\infty$-categories as claimed.
\end{proof}

\section{Prime fields in DGAs}\label{sect prime DG fields}
In this section, we discuss the notion of DG-division rings and DG-fields and obtain a classification of what we call the prime DG-fields. 

\begin{defi}\label{defi fields in DGAs}
    We say a DGA $A$ is a DG-division ring (DGDR) if $\pis A$ is a graded division ring in the sense that every non-zero homogeneous element in $\pis A$ is invertible. If furthermore $\pis A$ is a graded commutative ring, then we say $A$ is a DG-field (DGF). 
\end{defi}

The obvious examples of DGDRs are ordinary division rings like $\fp$ and $\mathbb{Q}$.

    Consider the DGA
    $F_{2n}^p:= F_{2n}(1) = S_{2n}[x_{2n}^{\pm 1}]$
    obtained from $S_{2n}:=S_{2n}^p$ by inverting the generator $x_{2n}$ and $F_\infty^p := \fp$. 
    By the universal property of localizations, for another DGA $A$, the restriction map
    \[\Map_{\Alg(\Z)}(F_{2n}^p,A) \to \Map_{\Alg(\Z)}(S_{2n}, A)\]
    is the inclusion of those components corresponding to maps $S_{2n} \to A$ carrying $x_{2n}$ to an invertible element, see \cite[Propositions 7.2.3.19 \& 7.2.3.27]{lurie2016higher}. 
    
By construction and definition, for all $1\leq n\leq \infty$ and all primes $p$, $F_{2n}^p$ is a DGF. However, much more is true, $F_{2n}^p$ are prime fields in the following sense.

\begin{defi}
    We say a DGDR $A$ is a prime DG-division ring if for every DGDR $B$, every map $B\to A$ of DGAs is an equivalence. If a prime DGDR $A$ has graded commutative homology, we call it a prime DGF.
\end{defi}

 Indeed $\fp$ and $\mathbb{Q}$ are examples of prime DG-fields (since every map of DGDRs is injective in homology).

\begin{prop}\label{prop f2n are prime}
    Each $F_{2n}^p$ is a prime DGF.
\end{prop}
\begin{proof}
    Let $A\to F_{2n}^p$ be a map of DGAs where $A$ is a DGDR. Since it is a map of division rings, $\pis A \to \pis F_{2n}^p$ is injective.  From this, we see that $\pis A \cong \fp[x_{2nl}^{\pm 1}]$ for some $l \geq 1$ (for here, $l = \infty$ case meaning $A \simeq \fp$). Therefore, it is sufficient to show that $l = 1$. Assume $l>1$, we take connective covers and Postnikov sections to obtain a map 
    \[\fp \simeq \tau_{\leq 2n}\tau_{\geq 0} A \to  \tau_{\leq 2n}\tau_{\geq 0} F_{2n}^p \simeq \tau_{\leq 2n} S_{2n}.\]
    This contradicts the fact that $\tau_{\leq 2n} S_{2n}$ is non-formal (Corollary \ref{coro:truncations-of-S-not-formal}) due to Corollary \ref{cor:map-implies-formal}.
\end{proof}

In fact, we already mentioned all examples of prime DGDRs. 
\begin{coro}[Corollary \ref{CorF}]
    The collection of all prime DGDRs is given by 
    \[ \{ \Q,\fp, F_{2n}^p \;|\; p \  \textrm{prime and } n>0 \}.\]
    Every DGDR receives a map from at least one of the prime DGDRs. Furthermore, every DGDR with even homology receives a map from exactly one of the prime DGDRs and this map is unique up to homotopy.
\end{coro}

\begin{proof}
     It follows from Proposition \ref{prop f2n are prime} that the DGAs given above are all DG-prime fields.  
    
    We first show that for a given a DGDR $A$, there is a DGA-map $B\to A$ from one of the DGAs given above. If $\pi_0A$ has characteristic $0$, then $A \simeq \mathbb{Q} \otimes_{\z} A$ is a $\mathbb{Q}$-algebra and it receives a map from $\mathbb{Q}$. 
    
    Assume that $\pi_0A$ has characteristic $p$. We need to show that there exists an $n$ and a map $F_{2n}^p \to A$. Since $A$ is a DGDR, for $n>0$, such maps are the same as maps $S_{2n} \to A$ which are non-trivial on $\pi_{2n}$. Since $A$ has characteristic $p$, we know that there is a map $S_2 \to A$. Now, either this map is non-trivial on $\pi_2$, in which case we are done, or it is trivial, in which case we obtain a map $S_4 \to A$. Repeating this argument eventually yields a map $F_{2n}^p \to A$ or a map $\fp \simeq \colim_n S_{2n} \to A$.

    For the first statement, let $B$ be a prime DG-division ring. Then $B$ receives a map from one of the DGAs listed in the theorem (as we just proved) but this implies that $B$ is equivalent to that DGA as $B$ is prime. As we already proved that every DGDR receives a map from one of the DGAs listed, this also proves the second statement.

     Now we prove the third statement.  If $\pi_0A$ has characteristic $0$, then $A$ is a $\mathbb{Q}$-algebra so it receives a unique DGA map from $\mathbb{Q}$ and no maps from the other prime DGDRs as they have finite characteristic.
    If $\pi_0A$ has characteristic $p$, then we already proved that there is a map $F_{2n}^p \to A$  for some $1\leq n\leq \infty$. Assume that there is another map $F_{2n'}^p \to A$ for some $n'$. Assume $n'>n$, this would provide two maps $S_{2n}\to A$ one factoring through $F_{2n}^p$ and the other factoring through $F_{2n'}^p$. This first sends $x_{2n}$ to a non-trivial element and the second sends it to a trivial element. This contradicts the uniqueness of maps $S_{2n}\to A$ given by Corollary \ref{coro s2k maps are unique up to homotopy}. 

    The up to homotopy uniqueness of this map follows by the universal property of localizations and Corollary \ref{coro s2k maps are unique up to homotopy} and the $m=p$ case of Corollary \ref{prop map of dgas fp to a}.
\end{proof}

\section{Applications to algebraic $K$-theory}\label{sect motivic pullback}
 First, we recall the terminology from \cite{land2023kthrypushouts} that a square of ring spectra is called a motivic pullback square, if it is sent to a pullback by any localizing invariant. For instance, by \cite[Corollary 4.28]{land2023kthrypushouts}, there is a motivic pullback square
\begin{equation}\label{eq coordinate axis mtvc pullback}
    \begin{tikzcd}
        \z[x,y]/xy \ar[r] \ar[d] & \z[y] \ar[d] \\
        \z[x] \ar[r] & \z[t_2].
    \end{tikzcd}
\end{equation}
with  $\lv x \rv = \lv y \rv = 0$ and $\lv t_2 \rv  = 2$. This is a diagram of $\bE_1$-$\z[x,y]$-algebras and 
\[\z[t_2] \simeq \z[x] \amalg_{\z[x,y]} \z[y]\]
where the pushout may be calculated in the category of $\bE_1$-$\Z[x,y]$-algebras.

Now, given a commutative ring $R$ and $x,y \in R$, we equip $R$ with a $\z[x,y]$-algebra structure in the evident way. If $(x,y)$ forms a regular sequence in $R$, applying the base change functor $-\otimes_{\z[x,y]} R$ gives another motivic pullback square 
\begin{equation}\label{eq quotient motivic pullback square}
    \begin{tikzcd}
    R/xy \ar[r] \ar[d] & R/x \ar[d] \\
    R/y \ar[r] & \odot
\end{tikzcd}
\end{equation}
where the DGA $\odot$ is similarly given by the pushout $ R/y \amalg_{R} R/x$ of $R$-DGAs and $\pis(\odot) \cong R/(x,y)[t]$, see \cite[Lemma 4.30]{land2023kthrypushouts}. 

In loc.\ cit.\ it was observed that in this situation, the DGA $\odot$ may well not be formal and it was noted that it would be interesting to find sufficient conditions that it is formal (as a ring spectrum). Here, we use our earlier results to give some cases where formality can indeed be shown and thereby obtain new relative algebraic $K$-theory computations. For the formal DGA $\fp[t_2]$, the $K$-theory $K_*(\fp[t_2])$ is computed in \cite{BM} and independently in \cite[Example 4.29]{land2023kthrypushouts} in terms of $\mathbb{W}_n(-)$, the ring of big Witt vectors of length $n$, see e.g.\ \cite[\S 1]{Hesselholt}.

\begin{coro}[Corrolary \ref{CorG}]\label{cor:first-example-formality}
    Consider the ring $\Z[X]$ with the two elements $X$ and $m$. Then the ring $\odot$ associated to the above situation is the formal DGA $\Z/m[t_2]$. In particular, for a prime $p$, we have
    \begin{equation*}
    \begin{split}
        K(\Z[X]/pX) \simeq& \ K(\Z) \oplus \fib(K(\fp) \to K(\fp[t_2])), \,\textrm{and} \\
        K_{2r}(\z[X]/pX) \cong &K_{2r}(\Z)\oplus \mathbb{W}_r(\fp), \textrm{\ and \ }   K_{2r+1}(\z[X]/pX) \cong K_{2r+1}(\Z).
    \end{split} 
    \end{equation*}
\end{coro}
\begin{proof}
    There is a DGA map $\Z/m \to \Z/m[X] \to \odot$. Hence we may appeal to Corollary~\ref{cor:map-implies-formal}. We conclude that there is a pullback square
    \[\begin{tikzcd}
        K(\Z[X]/mX) \ar[r] \ar[d] & K(\Z) \ar[d] \\
        K(\Z/m[X]) \ar[r] & K(\Z/m[t_2])
    \end{tikzcd}\]
    The ``in particular'' follows from observing that the top horizontal map splits and that the canonical map $K(\fp) \to K(\fp[X])$ is an equivalence and using \cite{BM} or \cite[Example 4.29]{land2023kthrypushouts}. 
\end{proof}

\begin{rema}
    One can think of $\Z[X]/pX$ as half arithmetic and half geometric coordinate axes, as it is the pullback of $\Z$ and $\fp[X]$ over $\fp$. 
\end{rema}

\begin{rema}
For a perfect field $k$ of characteristic $p$, one obtains $K_*(W(k)[x]/px,W(k))$ using the same methods. For this, we note that one may replace $\z$ in $\eqref{eq coordinate axis mtvc pullback}$ with $W(k)$ using \cite[Proposition 2.17]{land2023kthrypushouts}. Formality of the resulting $W(k)$-DGA $\odot$ follows by Remark \ref{rema generalizing to witt vectors}.
\end{rema}

\begin{rema}\label{rema motivic square with grading}
    Following the discussion in \cite[Example 4.10]{burklund2023kthryofcoconnectiverings} we may let  $x$ have an arbitrary positive even degree in \eqref{eq coordinate axis mtvc pullback} (and $|y|=0$) in which case one finds $\lv t \rv = \lv x \rv+2$. Furthermore, we can take the pushout defining $\z[t]$ at the level of $\z$-graded $\z[x,y]$-algebras with $x$ and $y$ of weight $1$ and $0$ respectively. In this situation, $t$ is also of weight $1$ as this is the only weight that allows for the compatibility of the pushout defining $\z[t]$ with the shearing functor considered in \cite[Example 4.10]{burklund2023kthryofcoconnectiverings}. In this situation, we write $x_{2k}$ for $x$ where $\lv x \rv = 2k$ and $t_{2k+2}$ for $t$.
\end{rema}

Let $f\co \z[y] \to \z \to \z[X_{2k}]$ denote the composite of the map of $\z$-graded $\bE_\infty$ $\z$-algebras carrying $y$ to $m$ and the unit map of $\z[X_{2k}]$; here, $m>0$ as before. We consider $\z[X_{2k}]$ as a $\z[x_{2k},y]$-algebra through the composite $\z$-graded $\bE_\infty$-map
\[\z[x_{2k},y] \xrightarrow{id \otimes f} \Z[X_{2k}] \otimes_{\z} \z[X_{2k}] \to \z[X_{2k}]\]
where  the last map is the multiplication map. Applying the base change functor  $- \otimes_{\z[x_{2k},y]} \z[X_{2k}]$ to the motivic pullback square mentioned in Remark \ref{rema motivic square with grading}, we obtain the following motivic pullback square. 

\begin{coro}
    Let $k>0$, then there is a motivic pullback square 
    \begin{equation*}
        \begin{tikzcd}
            \z[X_{2k}]/mX_{2k} \ar[r] \ar[d] & \ar[d] \z\\
            \z/m[X_{2k}] \ar[r] & \z/m[t_{2k+2}],
         \end{tikzcd}
    \end{equation*}
    where each entry above denotes the corresponding formal DGA and the left vertical map is the canonical map between formal DGAs that carries $X_{2k}$ to $X_{2k}$.
\end{coro}
\begin{proof}
    We need to show that the motivic pullback square constructed before the corollary is as stated. The bottom left corner is given by the $\z$-graded DGA $\z[x_{2k}] \otimes_{\z[x_{2k},y]} \Z[X_{2k}]$. As $\z[x_{2k}] \simeq \z[x_{2k},y]/y$,   the homology of this DGA is given by $\z/m[X_{2k}]$ with $X_{2k}$ of weight $1$.  In particular, it receives a map from $\z/m$ given by the inclusion of the weight $0$  component (see \cite[\S 2.2]{ausoni2022adjroot}) which shows that this DGA is formal by Corollary \ref{cor:map-implies-formal}.

   The bottom right corner of the motivic pullback square is given by $\z[t_{2k+2}] \otimes_{\z[x_{2k},y]} \z[X_{2k}]$ and a simple Tor computation ensures that the homotopy ring of this DGA is given by $\z/m[t_{2k+2}]$. Furthermore, the composite of $\z/m\to \z/m[X_{2k}]$ with the bottom horizontal map implies that this DGA is formal as desired (Corollary \ref{cor:map-implies-formal}). 

    The top left corner of this motivic pullback square  is given by the (homotopy) pullback  of DGAs 
   \[\z/m[X_{2k}] \times_{\z/m} \z.\]
We need to show that this is the formal DGA $\z[X_{2k}]/mX_{2k}$. The long exact sequence corresponding to this pullback shows that its homotopy ring is given by $\z[X_{2k}]/mX_{2k}$.   There are canonical DGA maps $\z[X_{2k}]/mX_{2k} \to \z/m[X_{2k}]$ and $\z[X_{2k}]/mX_{2k} \to \z$ and since there is an up-to homotopy unique map of DGAs $\z[X_{2k}]/mX_{2k} \to \z/m$, these maps lift to a map to the pullback above which can be seen to be an isomorphism in homology as desired. 
\end{proof}

\begin{rema}
    All the maps in this motivic pullback square are DGA maps. The only mysterious map here is the bottom horizontal map which we do not expect to identify with the canonical map between the corresponding formal DGAs; we do not pursue this matter here. However, the authors and Tamme are planning to compute the algebraic $K$-theory of the formal DGA $\fp[X_{2k}]$ for $k>0$ generalizing the main result of \cite{BM} or equivalently of \cite[Example 4.29]{land2023kthrypushouts}, giving the relevant computation of $K(\odot)$ in the above example. 
    \end{rema}

Let us now give a generalization of the Corollary \ref{cor:first-example-formality} in a different direction; we will use a special case of \eqref{eq quotient motivic pullback square} but we need to clarify the gradings we have. We begin with the motivic pullback square \eqref{eq coordinate axis mtvc pullback}. We consider the gradings mentioned in Remark \ref{rema motivic square with grading} (with $\lv x \rv = 0$) but in $\z/l$-grading in the canonical way, (i.e.\ by left Kan extending through the canonical surjection $\z \to \z/l$), see \cite[\S 2.2]{ausoni2022adjroot}. Furthermore, we consider the ring $\Z[X]$ with the two elements $X$ and $f$ where $f(X) = g(X^l)$ is a polynomial in $X^l$, for some $l\geq 1$, with constant term $f(0) = p$.  By placing $X$ in weight $1$, 
we equip $\Z[X]$ with a  $\Z/l$-grading; in this way, $f$ is of weight $0$ and $\z[X]$ is an algebra over $\z[x,y]$ (in $\z/l$-graded $\z$-modules) where $x$ and $y$ act through $X$ and $f$ respectively. Extending scalars through $- \otimes_{\z[x,y]} \z[X]$ (on \eqref{eq coordinate axis mtvc pullback}), we obtain the motivic pullback square \cite[Proposition 2.17]{land2023kthrypushouts}: 
\[ \begin{tikzcd}
    \Z[X]/Xf \ar[r] \ar[d] & \Z \ar[d] \\
    \Z[X]/f \ar[r] & \odot.
\end{tikzcd}\]
The DGA $\odot\simeq \z[t_2] \otimes_{\z[x,y]}\z[X]$ has homology $\fp[t_2]$ with $t_2$ in weight $1$. Hence the grading $0$ piece $\gr_0(\odot)$ of $\odot$ has homotopy ring given by $\fp[t_2^l]$. Consequently, if $l\geq p-1$,  $\gr_0(\odot)$ is topologically formal due to Theorem \ref{ThmB}. Therefore, we have the composite map  
\[\fp \to \gr_0(\odot) \to \odot\]
of ring spectra; the last map above is the inclusion of the zero component (see \cite[\S 2.2]{ausoni2022adjroot}).  Applying Corollary~\ref{cor:map-implies-formal}, we deduce that $\odot$ is equivalent, as a ring spectrum, to $\fp[t_2]$. In particular, we find:
\begin{coro}\label{coro mtvc quotients with general polynomials}
Let $f \in \Z[X]$ be a polynomial in $X^l$ with constant term $p$. If $l\geq p-1$, there is a pullback diagram
\[\begin{tikzcd}
    K(\Z[X]/Xf) \ar[r] \ar[d] & K(\Z) \ar[d] \\
    K(\Z[X]/f) \ar[r] & K(\fp[t_2]).
\end{tikzcd}\]
\end{coro}

\begin{rema}
    In the situation described above, the ring $\gr_0(\Z[X]/f)$ is isomorphic to $\Z[X]/g$ and hence need not be an $\fp$-algebra, contrary to the situation in Corollary~\ref{cor:first-example-formality} and we really do need to investigate $\gr_0(\odot)$ instead. Moreover, the assumption that $l\geq p-1$ cannot be relaxed too much: For instance, if $l=1$, we may consider the case $f=X+p$. In this case, the resulting ring $\odot$ is given by $S_2^p = \Z\sslash p$ \cite[Example 4.31]{land2023kthrypushouts}, which for $p$ odd is not formal as a ring spectrum.
\end{rema}

Finally, we consider motivic pullback square associated to the Rim square \cite[Example 4.32]{land2023kthrypushouts}.
\begin{equation}\label{eq mtvc rim square}
\begin{tikzcd}
    \Z[C_p] \ar[r] \ar[d] & \Z[\zeta_p] \ar[d] \\
    \Z \ar[r] & \odot
\end{tikzcd}
\end{equation}
This is \eqref{eq quotient motivic pullback square} with $R = \z[v]$ and chosen elements $v-1$ and $1+v +\cdots v^{p-1}$. The resulting DGA $\odot$ is $\Z[\zeta_p]\sslash (\zeta_p-1)$ or equivalently, as was shown in \cite{land2023kthrypushouts} by comparing to a construction of Krause--Nikolaus, $\tau_{\geq0}\Z^{tC_p}$. In this case, there is an equivalence of ring spectra $\odot \simeq \fp[t_2]$ \cite[Example 4.32]{land2023kthrypushouts}. In the following, $\Phi_{p^l}(X)$ denotes the $p^l$ cyclotomic polynomial; we have  $\z[\zeta_{p^l}] \cong \z[X]/\Phi_{p^l}(X)$  where $\Phi_{p^l}(X) = \Phi_{p}(X^{p^{l-1}})$ and $\Phi_p(Y) = 1+Y + \dots + Y^{p-1}$.

\begin{coro}\label{coro motivic for second cyclotomic extensions}
Let $0\leq k<l$, then there is a motivic pullback square:
\begin{equation*}
        \begin{tikzcd}
            \z[X](X^{p^k}-1)\Phi_{p^l(X)} \ar[d]\ar[r] & \z[\zeta_{p^l}] \ar[d]\\
            \z[X]/(X^{p^k}-1) \ar[r] & \fp[t_2].
        \end{tikzcd}
    \end{equation*}
\end{coro}
\begin{proof}
 This is the  motivic pullback square in \eqref{eq quotient motivic pullback square} with $R= \z[X]$ and the chosen elements $(X^{p^{k}}-1)$ and  $\Phi_{p^l}(X)$. Since $\Phi_{p^l}(X) = \Phi_{p}(X^{p^{l-1}})$, $\Phi_p(X) = 1+ X + \cdots +X^{p-1}$ and since $\z[\zeta_{p^l}]$ is a domain and $\zeta_{p^l}^{p^k}\not = 1$ in $\z[\zeta_{p^l}]$, we find that $\Phi_{p^l}(X)$ and $(X^{p^{k}}-1)$ form a regular sequence as desired.  Therefore, this provides the stated motivic pullback square except for the identification of $\odot$ with $\fp[t_2]$ as a ring spectrum.

   Again by the discussion on \eqref{eq quotient motivic pullback square}, $\pis \odot \cong\fp[t_2]$ and there is an equivalence of DGAs 
    \[\odot \simeq \z[X]/(X^{p^k}-1)\amalg_{\z[X]} \z[\zeta_{p^l}].\]

     We consider the commuting diagram of rings:
\begin{equation*}
\begin{tikzcd}
    \z \ar[d]& \z[v]\ar[r] \ar[d] \ar[l] & \z[\zeta_p]\ar[d]\\
    \z[X]/(X^{p^k}-1) & \ar[l]  \z[X] \ar[r]& \z[\zeta_{p^l}],  
\end{tikzcd}
\end{equation*}
where the middle vertical map carries $v$ to $X^{p^{k}}$, the map $\z[v] \to \z$ carries $v$ to $1$ and the top horizontal map on the right hand side is the quotient map to  $\z[v]/\Phi_p(v) \cong \z[\zeta_p]$. This gives a map of DGAs 
    \[\z \amalg_{\z[v]} \z[\zeta_{p}] \to \z[X]/(X^{p^k}-1) \amalg_{\z[X]} \z[\zeta_{p^l}]\simeq \odot.\]
     The first DGA above is the circle dot for the motivic square in \eqref{eq mtvc rim square} and as stated above, it is topologically equivalent to $\fp[t_2]$. Precomposing the DGA map above with a map  of ring spectra $\fp \to \z \amalg_{\z[v]} \z[\zeta_{p}]$, we deduce that $\odot$ is also topologically formal (Corollary \ref{cor:map-implies-formal}).

\end{proof}

The $k=l-1$ case of this corollary  generalizes the motivic pullback square in \eqref{eq mtvc rim square} as follows.

\begin{coro}[Corollary \ref{coro I}]\label{coro mtvc generalizing the rim square}
There is a motivic pullback square
 \begin{equation*}
        \begin{tikzcd}
            \z[C_{p^l}] \ar[d]\ar[r] & \z[\zeta_{p^l}] \ar[d]\\
            \z[C_{p^{l-1}}] \ar[r] & \fp[t_2].
        \end{tikzcd}
    \end{equation*}
\end{coro}

\subsection{On $\bE_\infty$-structures on $S_{2n}^p$}
We finish this paper with an observation about $\bE_\infty$-structures on the DGAs $S_{2n}:= S_{2n}^p$, possibly of independent interest. We consider the $\bE_\infty$-$\Z$-algebras $\z^{tC_p}$ and $\Z_{(p)}^{t\Sigma_p}$ as DGAs. We note that the inclusion $C_p \subseteq \Sigma_p$ induces an $\bE_\infty$-$\Z$-algebra map $\Z_{(p)}^{t\Sigma_p} \to \Z^{tC_p}$, which induces the map $\fp[u_{2p-2}^{\pm 1}] \to \fp[u_2^{\pm 1}]$ sending $u_{2p-2}$ to $u_2^{p-1}$ on homotopy groups. 

\begin{prop}\label{prop:truncation-of-tate-cp-not-formal}
    The DGA $\tau_{[0,2p-2]}\Z^{tC_p}$ is not formal.
\end{prop}
\begin{proof}
    Applying $\hh^\Z(-)\otimes_{\z}\fp$ to the  motivic pullback square \eqref{eq mtvc rim square}, we obtain a fibre sequence 
    \[ \hh^\Z(\Z[C_p])\otimes_{\z}\fp \to \big(\z \oplus \hh^\Z(\Z[\zeta_p])\big)\otimes_{\z} \fp \to \hh^\Z(\tau_{\geq0}\Z^{tC_p})/p\]
    or equivalently
    \[ \hh^{\fp}(\fp[C_p]) \to \fp \oplus \hh^{\fp}(\fp[\zeta_p]) \to \hh^\Z(\tau_{\geq 0}\Z^{tC_p})/p\]
    where $\fp[\zeta_p]$ is notation for $\Z[\zeta_p] \otimes_\Z \fp$. Recalling that $\fp[C_p]\cong \fp[x]/(x^p-1)$ and that $\fp[\zeta_p] \cong \fp[x]/(1+\cdots+x^{p-1})$, the results of \cite[pg.\ 54]{BuenosAires} apply to give an exact sequence of $\fp$-vector spaces
    \[ \dots \to \fp^{\oplus p} \to \fp^{\oplus p-2} \to \pi_{2p-1}(\hh^\Z(\tau_{\geq0}\Z^{tC_p})/p) \to \fp^{\oplus p} \to  \cdots \]
    showing that the middle term has $\fp$-dimension at most $2p-2$. Since $\hh^\Z(-)/p$ preserves connectivity, we find that the map
    \[ \hh^\Z(\tau_{\geq0}\Z^{tC_p})/p \to \hh^\Z(\tau_{[0,2p-2]}\Z^{tC_p})/p\]
    is an isomorphism in homotopical degrees $\leq 2p-1$ as  $\tau_{[0,2p-2]}\Z^{tC_p} \simeq \tau_{[0,2p-1]}\Z^{tC_p}$. The same is true for $\fp[u_2] \to \tau_{\leq 2p-2}\fp[u_2]$ in place of $\tau_{\geq 0}\Z^{tC_p}\to \tau_{[0,2p-2]}\Z^{tC_p}$. Therefore, it suffices to show that $\pi_{2p-1}(\hh^\Z(\fp[u_2])/p)$ has $\fp$-dimension larger than $2p-2$.  
    
    Additively, we have:
    \begin{equation*}
    \begin{split}
         \pis(\hh^\Z(\fp[u_2])/p) \cong& \pis \big(\hh^{\z}(\fp) \otimes_{\z} \hh^{\z}(\z[u_2])/p\big ) \\
         \cong &\fp[x_2,u_2]\otimes_{\fp} \Lambda_{\fp}[e_3,f_1] .
         \end{split}
    \end{equation*}
   The first equivalence follows since $\hh^{\z}(-)$ is symmetric monoidal; the second follows by standard computations and by noting that applying $\pis (-/p)$ on an $\fp$-module corresponds to applying $\pis (-) \otimes_{\fp}\Lambda_{\fp}(f_1)$. An $\fp$-basis of the degree $2p-1$ part is then given by $x_2^iu_2^{p-1-i}f_1$, with $i=0,\dots, p-1$ and $x_2^iu_2^{p-2-i}e_3$ with $i=0,\dots,p-2$. This shows that $\pi_{2p-1}(\hh^\Z(\fp[u_2])/p)$ has $\fp$-dimension $2p-1$ which is larger than $2p-2$ as desired.
\end{proof}

 \begin{coro}\label{cor:truncation-of-tate-sigmap-not-formal}
     The DGA $\tau_{[0,2p-2]} \Z_{(p)}^{t\Sigma_p}$ is not formal.
 \end{coro}
 \begin{proof}
     As noted earlier, there is a map $\tau_{[0,2p-2]} \Z_{(p)}^{t\Sigma_p} \to \tau_{[0,2p-2]}\Z^{tC_p}$, so if the domain is formal, we in particular obtain a map $\fp \to \tau_{[0,2p-2]}\Z^{tC_p}$ which, by Corollary~\ref{cor:map-implies-formal} contradicts Corollary~\ref{prop:truncation-of-tate-cp-not-formal}.
 \end{proof}

As a consequence of the uniqueness result we proved in Theorem~\ref{ThmD}, we obtain:
\begin{coro}\label{cor:S2p-2=tatesigma}
    The unique map $S_{2p-2} \to \tau_{\geq0}\Z_{(p)}^{t\Sigma_p}$ is an equivalence of DGAs.
\end{coro}

\begin{rema}
    As a consequence of Corollary~\ref{cor:S2p-2=tatesigma}, we find that $S_{2p-2}$ admits an $\bE_\infty$-$\Z$-algebra structure. By the Hopkins--Mahowald theorem \cite[Theorem 5.1]{barthel}, there is in particular a map of ring spectra $\fp \to S_{2p-2}$. Together with Corollary~\ref{cor:map-implies-formal}, this gives another proof of the topological formality of $S_{2n}$ for $n\geq p-1$.
\end{rema}

\begin{rema}\label{remark:tate-fp}
    As a consequence of Corollary~\ref{cor:S2p-2=tatesigma} we have the equivalence of $\bE_1$-$\fp$-algebras.
    \[ \fp[\Omega\CP^{p-1}] \simeq \fp\otimes_\Z S_{2p-2} \simeq \fp\otimes_\Z \tau_{\geq0}\Z_{(p)}^{t\Sigma_p} \simeq \tau_{\geq0}\fp^{t\Sigma_p}\]
\end{rema}

Earlier, in Remark \ref{remark:non-unqieness}, we fixed $\z[X_{2n}]$-algebra structures on $S_{2n}$ through which we defined $S_{2nl}[\sqrt[l]{x_{2nl}}]$ in Construction \ref{cons adj root to snm}. Since we now know that $S_{2p-2}$ is an $\bE_\infty$-$\Z$-algebra by Corollary \ref{cor:S2p-2=tatesigma}, we can choose an $\bE_2$-$\z$-algebra map $\z[X_{2p-2}]\to S_{2p-2}$  \cite[Proposition 3.15]{ausoni2022adjroot} which provides a possibly different $\z[X_{2p-2}]$-algebra structure on $S_{2p-2}$ than the one we fixed earlier. Through this, we obtain (again a possibly different) $S_{2p-2}[\sqrt[p-1]{x_{2p-2}}]$ through Construction \ref{cons adj root to snm}.  

\begin{coro}
        For $S_{2p-2}[\sqrt[p-1]{x_{2p-2}}]$ as  above, there is an equivalence of DGAs
        \[S_{2p-2}[\sqrt[p-1]{x_{2p-2}}] \simeq \tau_{\geq 0} \Z^{tC_p}.\] 
\end{coro}
\begin{proof}
    By Corollary \ref{cor:S2p-2=tatesigma} and the discussion above, the claim will follow once we show that there is an equivalence \[\tau_{\geq0}\Z_{(p)}^{t\Sigma_p}[\sqrt[p-1]{u_{2p-2}}] \simeq \tau_{\geq0}\Z^{tC_p}.\]Since we started with an $\bE_2$-map $\z[X_{2p-2}]\to \tau_{\geq0}\Z_{(p)}^{t\Sigma_p}$,
    \[\tau_{\geq0}\Z_{(p)}^{t\Sigma_p}[\sqrt[p-1]{u_{2p-2}}]:= \tau_{\geq0}\Z_{(p)}^{t\Sigma_p} \otimes_{\z[X_{2p-2}]}\z[X_2] \]
     admits the structure of a $\tau_{\geq0}\Z_{(p)}^{t\Sigma_p}$-algebra. Upon inverting $u_{2p-2}$, we have two $\Z_{(p)}^{t\Sigma_p}$-algebras $\tau_{\geq0}\Z_{(p)}^{t\Sigma_p}[\sqrt[p-1]{u_{2p-2}}][u_{2p-2}^{\pm 1}]$ and $\Z^{tC_p}$ whose homotopy rings are isomorphic as $\pis \Z_{(p)}^{t\Sigma_p}$-algebras. Furthermore, their homotopy rings are  \'etale over $\pis \Z_{(p)}^{t\Sigma_p}$. It follows by \cite[Theorem 1.10]{hesselholt2022dirac}  that these two $\Z_{(p)}^{t\Sigma_p}$-algebras are equivalent. Taking connective covers gives the desired equivalence $\tau_{\geq0}\Z_{(p)}^{t\Sigma_p}[\sqrt[p-1]{u_{2p-2}}] \simeq \tau_{\geq0}\Z^{tC_p}$ of DGAs.
\end{proof}
It follows by Corollary \ref{cor:S2p-2=tatesigma} that $S_{2p-2}$ can be refined to an $\bE_\infty$-DGA. In fact, we conjecture below that for all $n\geq 1$, $S_{2p^l-2}$ can be refined to an $\bE_\infty$-DGA. We thank Oscar Randal-Williams for pointing out the following:
\begin{lemm}\label{lemm einfty structures on s2n}
    $\fp[\Omega\CP^{k}]$ refines to an $\bE_\infty$-$\fp$-algebra if $k=p^l-1$ and does not refine to an $\bE_2$-$\fp$-algebra if $k\neq p^l-1$.
\end{lemm}
\begin{proof}
    For every $n\geq 1$, using Dunn-additivity, there is the Bar-Cobar adjunction
    \[ \begin{tikzcd}
        \Alg_{\bE_n}^{\mathrm{aug}}(\fp) \simeq \Alg_{\bE_1}^{\mathrm{aug}}(\Alg_{\bE_{n-1}}(\fp)) \ar[r, shift left, "\Barc"] & \mathrm{CoAlg}^{\mathrm{coaug}}(\Alg_{\bE_{n-1}}(\fp)) \ar[l,shift left, "\Cobar"]
    \end{tikzcd}\]
    Since $\fp[\Omega \CP^k]$ is, as an augmented $\fp$-algebra, connected and finite, Bar-Cobar duality gives an equivalence of $\bE_1$-$\fp$-algebras
    \[ \Cobar(\Barc(\fp[\Omega \CP^{k}])) \simeq \fp[\Omega \CP^{k}].\]
    It hence suffices to analyse when, as an $\bE_1$-$\fp$-coalgebra, $\Barc(\fp[\Omega \CP^{k}]) \simeq \fp[\CP^{k}]$ admits the structure of a (commutative) biaugmented bialgebra. By $\fp$-linear duality, this is in turn equivalent to analysing when the $\bE_1$-$\fp$-algebra $\Mor(\CP^{k},\fp)$, i.e.\ the usual $\fp$-valued cochain algebra of $\CP^k$, admits the structure of a (cocommutative) biaugmented bialgebra. Now we observe that this $\bE_1$-algebra is formal, i.e.\ $\Mor(\CP^k,\fp) \simeq \fp[x]/x^{k+1}$ for $|x|=-2$. This is for instance proven in \cite[Prop.\ 2.1]{Westerland}, the proof in loc.\ cit.\ applies in fact integrally.  A coproduct on $\fp[x]/x^{k+1}$ is determined by its effect on the element $x$, which for formal reasons must be $1\otimes x + x \otimes 1$ (and is in particular coassociative if it exists). This is a coproduct if and only if $(1\otimes x+x\otimes 1)^{k+1} = 0$.
    But
    \[ 0= (1\otimes x+x\otimes 1)^{k+1} = \sum\limits_{i=0}^{k+1}{k+1 \choose i}x^i\otimes x^{k+1-i} = \sum_{i=1}^k {k+1 \choose i} x^i\otimes x^{k+i-1}\]
    implies that all binomial coefficients have to vanish modulo $p$, and this can be shown to be the case if and only if $k+1=p^l$ as a consequence of Lucas' theorem.
\end{proof}

From the equivalence $\fp\otimes_\Z S_{2n} \simeq \fp[\Omega \CP^n]$, we deduce that the $\fp$-algebra $\fp \otimes_{\z} S_{2p^l-2}$ is $\bE_\infty$ for all $l$ and that for $n\neq p^l-1$, the DGA $S_{2n}$ does not refine to an $\bE_2$-DGA.

\begin{coro}\label{coro s2n are not always e2}
    Let $n \neq p^l-1$, then $S_{2n}$ does not admit the structure of an $\bE_2$-DGA.
\end{coro}

The evidence we have so far leads us to the following conjecture. 

\begin{conj}
    For each $l>0$, the DGA $S_{2p^l-2}$ admits the structure of an $\bE_\infty$-DGA.
\end{conj}

\begin{rema}
    We have observed in Remark~\ref{remark:tate-fp} that there is an equivalence of $\fp$-algebras $\fp[\Omega \CP^{p-1}] \simeq \tau_{\geq0} \fp^{t\Sigma_p}$. The target of this equivalence is an $\bE_\infty$-$\fp$-algebra, and we have just argued that the source also admits an $\bE_\infty$-structure.  We have no reason to believe that these two $\bE_\infty$-structures are equivalent, but do not pursue this matter here.
\end{rema}

\providecommand{\bysame}{\leavevmode\hbox to3em{\hrulefill}\thinspace}
\providecommand{\MR}{\relax\ifhmode\unskip\space\fi MR }
\providecommand{\MRhref}[2]{%
  \href{http://www.ams.org/mathscinet-getitem?mr=#1}{#2}
}
\providecommand{\href}[2]{#2}

\end{document}